\documentclass[11pt, a4paper]{article}
\usepackage[T1]{fontenc}
\usepackage{lmodern}
\usepackage{HongLiu}

\usepackage{chngcntr}
\counterwithout{figure}{section}

\hypersetup{hidelinks}

\usetikzlibrary{arrows}

\definecolor{commentpurple}{rgb}{0.5,0,0.5}

\newtheorem{problem}[theorem]{Problem}

\newtheorem*{remark}{Remark}

\numberwithin{equation}{section}

\title{\Large \bf Strong counterexamples to Mubayi's supersaturation conjecture in every uniformity}

\author{
Heng~Li\thanks{School of Mathematics, Shandong University, Jinan, China, and Extremal Combinatorics and Probability Group (ECOPRO), Institute for Basic Science (IBS), Daejeon, South
Korea. Email: \texttt{heng.li@sdu.edu.cn}.}
\and
Hong~Liu\thanks{Extremal Combinatorics and Probability Group (ECOPRO), Institute for Basic Science (IBS), Daejeon, South
Korea. Email: \texttt{hongliu@ibs.re.kr}.}
\and 
Xizhi~Liu\thanks{School of Mathematical Sciences, University of Science and Technology of China, Hefei, China. Email:~\texttt{liuxizhi@ustc.edu.cn}.}
\and
Jing~Wang\thanks{School of Mathematics and Statistics, Henan Normal University, Henan, China, and Extremal Combinatorics and Probability Group (ECOPRO), Institute for Basic Science (IBS), Daejeon, South
Korea. Email:~\texttt{wangjing2022@htu.edu.cn}.}
}

\date{\today}

\begin{document}
\maketitle

\begin{abstract}
The supersaturation problem asks, for a fixed $r$-graph $\mathcal F$, for the
minimum number of copies of $\mathcal F$ in an $n$-vertex $r$-graph with
$\ex(n,\mathcal F)+q$ edges.  Mubayi conjectured a local form of supersaturation
under a stability hypothesis: if $\mathcal F$ is non-$r$-partite and stable,
meaning roughly that the extremal $\mathcal F$-free construction is unique and
all near-extremal $\mathcal F$-free $r$-graphs are close to it, then this minimum
should be at least $q c(n,\mathcal F)$, where $c(n,\mathcal F)$ is the minimum
number of copies created by adding one edge to the extremal $\mathcal F$-free
$r$-graph.

We disprove this conjectured local lower bound in the strongest possible
quantitative sense that the ratio to the conjectured lower bound can be made
arbitrarily small. More precisely, for every
$r\ge2$ and every $K>1$, we construct a stable non-$r$-partite $r$-graph
$\mathcal F$ such that, for all sufficiently large $n$ and every
$1\le q\le\delta n$, there is an $n$-vertex $r$-graph with
$\ex(n,\mathcal F)+q$ edges containing at most
$K^{-1}q c(n,\mathcal F)$ copies of $\mathcal F$. This arbitrarily large
failure already occurs when $q=1$.
\end{abstract}



\section{Introduction}\label{introduction}

The Tur\'an problem, a central topic in extremal combinatorics, asks for the
maximum number of edges in an $n$-vertex graph, or more generally in an
$r$-uniform hypergraph, which avoids a fixed forbidden configuration.  In the
graph case, Mantel's theorem determines the extremal number for triangles, and
Tur\'an's theorem extends this to complete graphs $K_{r+1}$, with the balanced
$r$-partite Tur\'an graph as the extremal graph~\cite{Mantel1907,Turan1941}.

Supersaturation concerns what happens just beyond this extremal threshold.  The
classical theorem of Rademacher, recorded by Erd\H{o}s~\cite{Erdos1955}, says
that every graph with $\lfloor n^2/4\rfloor+1$ edges contains at least
$\lfloor n/2\rfloor$ triangles.  Subsequent work of Erd\H{o}s
\cite{Erdos1955,Erdos1962} and Lov\'asz--Simonovits
\cite{LovaszSimonovits1976,LovaszSimonovits1983} developed exact
supersaturation results for cliques.  The general supersaturation theorem of
Erd\H{o}s and Simonovits~\cite{ErdosSimonovits1983} shows that any fixed
positive excess over the Tur\'an density forces a positive density of copies of
the forbidden graph.  The question studied here is more local: when only $q$
edges are added above extremality, should the cheapest construction simply add
the $q$ cheapest new edges to the extremal example?

We now introduce the notation needed to state this local question.  Fix
$r\ge2$.  An \emph{$r$-uniform hypergraph}, or \emph{$r$-graph}, is a family of
$r$-subsets of a vertex set.  We identify a hypergraph with its edge set, and
hence write $|\mathcal H|$ for the number of edges of $\mathcal H$.  For a fixed
$r$-graph $\mathcal F$, its \emph{Tur\'an number} is
\[
        \ex(n,\mathcal F)
        :=
        \max\big\{|\mathcal H|:\ |V(\mathcal H)|=n
        \text{ and } \mathcal H \text{ is } \mathcal F\text{-free}\big\}.
\]
Here $\mathcal H$ is \emph{$\mathcal F$-free} if it contains no copy, not
necessarily induced, of $\mathcal F$.  Throughout this paper, a copy means a
labelled embedding.  For $r$-graphs $\mathcal F$ and $\mathcal G$, let
$N_{\mathcal F}(\mathcal G)$ denote the number of copies of $\mathcal F$ in
$\mathcal G$, and let $N_{\mathcal F}(\mathcal G;e)$ denote the number of such
copies in which $e$ is the image of some edge of $\mathcal F$.

Suppose that, for all sufficiently large $n$, the extremal $n$-vertex
$\mathcal F$-free $r$-graph is unique, and denote it by
$\mathrm{EX}(n,\mathcal F)$.  Define
\[
        c(n,\mathcal{F})
        :=
        \min\left\{
        N_{\mathcal{F}}(\mathrm{EX}(n,\mathcal F)\cup\{e\};e):
        e\in \binom{V(\mathrm{EX}(n,\mathcal F))}{r}
        \setminus \mathrm{EX}(n,\mathcal F)
        \right\},
\]
and
\[
        h_{\mathcal{F}}(n,q)
        :=
        \min\left\{
        N_{\mathcal F}(\mathcal H):
        |V(\mathcal H)|=n
        \text{ and }
        |\mathcal H|=\ex(n,\mathcal F)+q
        \right\}.
\]
Thus $c(n,\mathcal F)$ is the one-edge cost above the extremal construction,
whereas $h_{\mathcal F}(n,q)$ is the true minimum at excess $q$.

Mubayi's work on counting substructures
\cite{Mubayi2009,Mubayi2010,Mubayi2013} established local lower bounds of this
type in several important settings, including color-critical graphs and certain
non-$r$-partite $r$-graphs.  These results suggest the following natural
principle: if the extremal $\mathcal F$-free construction is sufficiently rigid,
then the first-order cost of adding $q$ edges should be at least $q$ times the
least possible cost of adding one edge.  This led to the following conjecture.

\begin{conjecture}[Mubayi \cite{Mubayi2013}]
\label{conj:Mubayi}
Let $\mathcal{F}$ be a stable non-$r$-partite $r$-graph. Then, for every positive integer $q$ and all sufficiently large $n$, every $n$-vertex $r$-graph with $\ex(n,\mathcal{F})+q$ edges contains at least $q\,c(n,\mathcal{F})$ copies of $\mathcal{F}$.
\end{conjecture}

Here, stable means that, for all sufficiently large $n$, the extremal $n$-vertex $\mathcal F$-free $r$-graph is unique and every $n$-vertex $\mathcal F$-free $r$-graph with $(1-o(1))\ex(n,\mathcal F)$ edges differs from it in only $o(n^r)$ edges.

For graphs, Ma and Yuan~\cite{MaYuan2025} recently constructed infinitely many counterexamples of arbitrary chromatic number at least four.  They also asked whether the one-edge equality $h_{\mathcal F}(n,1)=c(n,\mathcal F)$ holds for every graph $\mathcal F$ containing a cycle, and left open both the hypergraph case $r\ge3$ and the graph case of chromatic number three.  Subsequent graph constructions of Chen--Yuan~\cite{ChenYuanStrongMaYuan} and of the first author with Li and Ma (unpublished observations) gave further strong negative answers to the one-edge question, but those examples do not satisfy the stability hypothesis in Conjecture~\ref{conj:Mubayi}.

Thus, under the stability hypothesis in Conjecture~\ref{conj:Mubayi},
the one-edge case remained open, even for $3$-chromatic graphs. Even if
the conjectured bound were too optimistic, it was still
conceivable that stability might guarantee a weaker lower bound of the
correct order,
$h_{\mathcal F}(n,q)=\Omega_r\bigl(q\,c(n,\mathcal F)\bigr).$
Our main result rules out even this weaker possibility: for every fixed
uniformity, the ratio
\[
\frac{h_{\mathcal F}(n,q)}{q\,c(n,\mathcal F)}
\]
can be made arbitrarily small, already when $q=1$.

\begin{theorem}\label{thm:main}
For every integer $r\ge 2$ and every real number $K>1$, there is a stable
non-$r$-partite $r$-graph $\mathcal{F}$ with the following property. There exist constants
$\delta_{\mathcal{F}},n_{\mathcal{F}}>0$ such that, for every $n\ge n_{\mathcal{F}}$ and every integer $q$
with $1\le q\le \delta_{\mathcal{F}} n$, it holds that 
$h_{\mathcal{F}}(n,q)\le K^{-1}\,q c(n,\mathcal{F})$.
\end{theorem}

The forbidden hypergraph used in the proof is the semi-blowup fan $\mathcal F_{r,t}$, defined in Section~\ref{sec:prelim}.  For $r=2$, this gives a stable $3$-chromatic graph counterexample; for $r\ge3$, it gives the first hypergraph counterexamples of this type.  The parameter $t$ controls the strength of the failure: taking $t$ large makes the surviving proportion of one-new-edge copies smaller than any
prescribed constant factor.

To verify that these examples genuinely fall within the scope of
Conjecture~\ref{conj:Mubayi}, we need the exact Tur\'an theorem for
$\mathcal F_{r,t}$ and uniqueness of the extremal construction. The
following theorem supplies this missing exact statement. Let $\mathcal T_r^{(r)}(m)$ denote the complete balanced $r$-partite $r$-graph
on $m$ vertices, and write $t_r^{(r)}(m):=|\mathcal T_r^{(r)}(m)|$.  For
$s\ge0$, put
\[
        \mathcal J_s(n):=K_s^{(r)}\vee \mathcal T_r^{(r)}(n-s),
\]
where $K_s^{(r)}$ denotes the complete $r$-graph on an $s$-vertex set, with no
edges if $s<r$, and the \emph{join} adds every $r$-set meeting both sides.
Thus $\mathcal J_s(n)$ is obtained from $\mathcal T_r^{(r)}(n-s)$ by adjoining
$s$ full-degree vertices.

\begin{theorem}\label{thm:exact}
Fix integers $r\ge 2$ and $t\ge r$. For all sufficiently large $n$, the unique
extremal $n$-vertex $\mathcal{F}_{r,t}$-free $r$-graph is $\mathcal{J}_{t-1}(n)$. 
In particular,
\[
        \ex(n,\mathcal{F}_{r,t})=|\mathcal{J}_{t-1}(n)|
        =
        \tbinom nr-\tbinom{n-t+1}{r}+t_r^{(r)}(n-t+1).
\]
\end{theorem}

For $r\ge3$, Theorem~\ref{thm:exact} extends a result from
\cite{HouLiLiuYuanZhang2025}.  In their density Corr\'adi--Hajnal setting, the
same construction $\mathcal J_{t-1}(n)$ is the extremal $n$-vertex $r$-graph
with no $t$ vertex-disjoint copies of the fan $\mathcal F_r$.  The present
theorem forbids only the larger semi-blowup configuration $\mathcal F_{r,t}$,
rather than all $t$ disjoint fans, and still obtains the same unique extremal
construction.

We now describe the idea behind Theorem~\ref{thm:main}.  Start with the unique
extremal construction $\mathcal J_{t-1}(n)$.  Let $S$ be its set of $t-1$
full-degree vertices, and let $V_1\cup\cdots\cup V_r$ be the balanced
$r$-partition of the remaining vertices.  The only non-edges among the partite
vertices are the non-crossing $r$-sets.  We add new non-crossing edges inside a
small set $Z$ of partite vertices, but we simultaneously delete the $S$-star of
$Z$,
\[
        \mathcal S(Z)
        :=
        \bigl\{\{z\}\cup Y:\ z\in Z,\ Y\in \tbinom{S}{r-1}\bigr\}.
\]
The number of deleted edges is only linear in $|Z|$, and we compensate for these
deleted edges by adding the same number of extra partite non-edges, plus the
desired excess $q$.

The key point is that the deleted $S$-star is extremely efficient at destroying
copies created by a new edge.  If a copy of $\mathcal F_{r,t}$ uses exactly one
new partite non-edge, then one fan is active and uses this new edge.  The other
$t-1$ fans cannot live entirely in the $r$-partite part of
$\mathcal J_{t-1}(n)$, because each fan is non-$r$-partite.  Hence these
inactive fans are forced to use the vertices of $S$ bijectively.  The
semi-blowup transversal edges then impose strong compatibility conditions among
the inactive fans.  These conditions create many opportunities for the copy to
use an edge of $\mathcal S(Z)$, and once such an edge is needed, the copy is
destroyed by the deletion.  Repeating this across the $t-1$ inactive fans gives
an exponentially small upper bound on the proportion of one-new-edge copies that
survive.

There are two technical difficulties.  First, the construction must work already
for $q=1$, so the edge count has to be adjusted exactly rather than only
asymptotically.  Second, after many new non-crossing edges are added inside
$Z$, one must control copies using two or more new edges.  These copies are
shown to be lower order for $q\le\delta n$, and are absorbed by choosing
$\delta$ sufficiently small.  The novelty of the construction is that the
semi-blowup fan combines a rigid exact Tur\'an structure with a highly
concentrated local family of one-edge completions, allowing the star deletion
near $S$ to defeat the expected lower bound by an arbitrary constant factor.

The paper is organized as follows.  Section~\ref{sec:prelim} fixes notation and
defines $\mathcal F_{r,t}$.  Section~\ref{sec:ss} proves
Theorem~\ref{thm:main}.  Section~\ref{sec:exact} records the stability tools and
proves Theorem~\ref{thm:exact}.  Section~\ref{sec:remarks} ends with a short
discussion and an open problem.

\section{Preliminaries}\label{sec:prelim}

This section collects the notation used in both main arguments. We first set
up the standard operations on $r$-graphs and the join construction
$\mathcal J_s(n)$, then define the fan configurations.

For a positive integer $m$, write $[m]:=\{1,\dots,m\}$. %All asymptotic notation may depend on the fixed forbidden $r$-graph under discussion. 
For an $r$-graph $\mathcal{F}$ and a positive
integer $s$, the \emph{blowup} $\mathcal{F}[s]$ is obtained by replacing each vertex of $\mathcal{F}$ by
an independent cluster of size $s$ and each edge by the complete $r$-partite
$r$-graph across the corresponding clusters.

We write $v(\mathcal H):=|V(\mathcal H)|$. An $r$-graph is \emph{$r$-partite} if its vertex set can be partitioned into
$r$ parts so that every edge has one vertex in each part. With respect to
a fixed $r$-partition, an $r$-set is called \emph{crossing} if it has one vertex
in each part. We keep the notation $\mathcal T_r^{(r)}(m)$ for the complete balanced
$r$-partite $r$-graph on $m$ vertices.

For an $r$-graph $\mathcal H$ and a vertex set $X\subseteq V(\mathcal H)$, the
\emph{induced subgraph} $\mathcal H[X]$ has vertex set $X$ and edge set $\{e\in \mathcal H:e\subseteq X\}$. 
If $U\subseteq V(\mathcal H)$, we write
$\mathcal H-U:=\mathcal H[V(\mathcal H)\setminus U]$ for the induced $r$-graph
obtained by deleting the vertices of $U$.
If $\mathcal E$ is a family of $r$-sets, then
$\mathcal H-\mathcal E$ denotes the $r$-graph obtained from
$\mathcal H$ by deleting the edges in
$\mathcal E\cap \mathcal H$ and keeping the vertex set unchanged.

For $s\ge 0$, keep the notation $\mathcal J_s(n):=K_s^{(r)}\vee \mathcal T_r^{(r)}(n-s)$. 
We call these $s$ full-degree vertices \emph{exceptional} and the remaining vertices the
\emph{partite vertices}. If $S$ is the exceptional set and
$V_1\cup\cdots\cup V_r$ is the balanced partition of the partite vertices, then $\mathcal J_s(n)-S=\mathcal T_r^{(r)}(n-s)$ with parts $V_1,\ldots,V_r$. Thus every $r$-set meeting $S$ is an edge
of $\mathcal J_s(n)$. The non-edges contained in
$V(\mathcal J_s(n))\setminus S$ are precisely the non-crossing $r$-sets with
respect to this partition; we call them \emph{partite non-edges}.

We now introduce the forbidden configurations used throughout the paper.

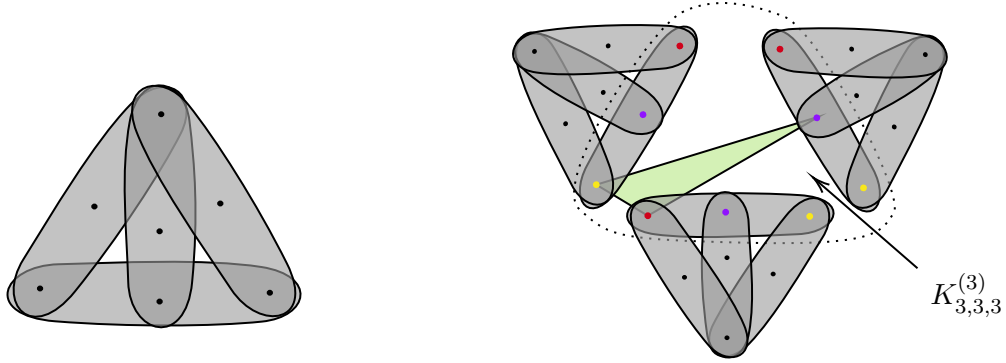
\begin{figure}[htbp]
    \centering

\tikzset{every picture/.style={line width=0.75pt}} %set default line width to 0.75pt        

\begin{tikzpicture}[x=0.60pt,y=0.60pt,yscale=-1,xscale=1]
%uncomment if require: \path (0,775); %set diagram left start at 0, and has height of 775

%Shape: Polygon [id:ds7941921722110875] 
\draw  [fill={rgb, 255:red, 184; green, 233; blue, 134 }  ,fill opacity=0.6 ] (406.75,186.02) -- (546.33,143.34) -- (439.06,205.48) -- cycle ;
%Shape: Regular Polygon [id:dp6040045821398934] 
\draw  [fill={rgb, 255:red, 155; green, 155; blue, 155 }  ,fill opacity=0.6 ] (446.86,96.87) .. controls (462.47,77.35) and (475.93,89.19) .. (467.22,108.47) .. controls (458.51,127.76) and (439.88,160.47) .. (421.48,187.34) .. controls (403.08,214.21) and (388.82,196.22) .. (400.98,173.89) .. controls (413.13,151.57) and (431.25,116.39) .. (446.86,96.87) -- cycle ;
%Shape: Regular Polygon [id:dp8062185346277478] 
\draw  [fill={rgb, 255:red, 155; green, 155; blue, 155 }  ,fill opacity=0.6 ] (357.69,112.6) .. controls (351.6,90.34) and (369.33,83.19) .. (379.46,99.13) .. controls (389.59,115.07) and (404.42,145.59) .. (414.7,173.44) .. controls (424.99,201.3) and (401.83,206.58) .. (391.37,186.26) .. controls (380.9,165.94) and (363.79,134.85) .. (357.69,112.6) -- cycle ;
%Shape: Polygon Curved [id:ds03140495195428916] 
\draw  [fill={rgb, 255:red, 155; green, 155; blue, 155 }  ,fill opacity=0.6 ] (363.6,115.93) .. controls (346.7,101.28) and (359.3,85.73) .. (376.69,93.06) .. controls (394.08,100.39) and (413.56,109.74) .. (437.14,126.63) .. controls (460.72,143.52) and (442.92,161.77) .. (423.01,151.08) .. controls (403.1,140.39) and (380.5,130.59) .. (363.6,115.93) -- cycle ;
%Shape: Regular Polygon [id:dp5721579599862564] 
\draw  [fill={rgb, 255:red, 155; green, 155; blue, 155 }  ,fill opacity=0.6 ] (369.14,115.66) .. controls (347.47,112.59) and (351.62,93.22) .. (370.4,90.03) .. controls (389.18,86.84) and (422.55,85.01) .. (451.17,86.38) .. controls (479.79,87.74) and (471.05,111.45) .. (448.5,113.22) .. controls (425.95,114.99) and (390.81,118.73) .. (369.14,115.66) -- cycle ;
%Shape: Ellipse [id:dp9016363446449527] 
\draw  [fill={rgb, 255:red, 0; green, 0; blue, 0 }  ,fill opacity=1 ] (387,148.61) .. controls (386.6,148.38) and (386.49,147.83) .. (386.75,147.38) .. controls (387.01,146.93) and (387.55,146.75) .. (387.95,146.98) .. controls (388.35,147.22) and (388.46,147.77) .. (388.2,148.22) .. controls (387.94,148.67) and (387.4,148.85) .. (387,148.61) -- cycle ;
%Shape: Ellipse [id:dp553581618178335] 
\draw  [fill={rgb, 255:red, 0; green, 0; blue, 0 }  ,fill opacity=1 ] (367.5,103.26) .. controls (367.1,103.03) and (366.99,102.47) .. (367.25,102.02) .. controls (367.51,101.57) and (368.04,101.4) .. (368.44,101.63) .. controls (368.84,101.86) and (368.96,102.41) .. (368.7,102.86) .. controls (368.44,103.31) and (367.9,103.49) .. (367.5,103.26) -- cycle ;
%Shape: Ellipse [id:dp14352112919776172] 
\draw  [fill={rgb, 255:red, 0; green, 0; blue, 0 }  ,fill opacity=1 ] (413.71,99.8) .. controls (413.31,99.56) and (413.2,99.01) .. (413.46,98.56) .. controls (413.72,98.11) and (414.25,97.94) .. (414.65,98.17) .. controls (415.05,98.4) and (415.17,98.95) .. (414.91,99.4) .. controls (414.65,99.85) and (414.11,100.03) .. (413.71,99.8) -- cycle ;
%Shape: Ellipse [id:dp14484114103499335] 
\draw  [fill={rgb, 255:red, 0; green, 0; blue, 0 }  ,fill opacity=1 ] (410.62,128.98) .. controls (410.22,128.75) and (410.1,128.2) .. (410.36,127.75) .. controls (410.62,127.3) and (411.16,127.12) .. (411.56,127.36) .. controls (411.96,127.59) and (412.07,128.14) .. (411.81,128.59) .. controls (411.55,129.04) and (411.02,129.22) .. (410.62,128.98) -- cycle ;
%Shape: Ellipse [id:dp6288908241413385] 
\draw  [color={rgb, 255:red, 248; green, 231; blue, 28 }  ,draw opacity=1 ][fill={rgb, 255:red, 248; green, 231; blue, 28 }  ,fill opacity=1 ] (406,187.32) .. controls (405.28,186.9) and (405.04,185.99) .. (405.45,185.27) .. controls (405.86,184.55) and (406.78,184.31) .. (407.5,184.72) .. controls (408.22,185.14) and (408.46,186.05) .. (408.05,186.77) .. controls (407.63,187.49) and (406.72,187.73) .. (406,187.32) -- cycle ;
%Shape: Ellipse [id:dp686737926728553] 
\draw  [color={rgb, 255:red, 144; green, 19; blue, 254 }  ,draw opacity=1 ][fill={rgb, 255:red, 144; green, 19; blue, 254 }  ,fill opacity=1 ] (435.27,143.39) .. controls (434.55,142.97) and (434.31,142.05) .. (434.72,141.34) .. controls (435.14,140.62) and (436.05,140.37) .. (436.77,140.79) .. controls (437.49,141.2) and (437.73,142.12) .. (437.32,142.84) .. controls (436.9,143.56) and (435.99,143.8) .. (435.27,143.39) -- cycle ;
%Shape: Ellipse [id:dp5042626474593314] 
\draw  [color={rgb, 255:red, 208; green, 2; blue, 27 }  ,draw opacity=1 ][fill={rgb, 255:red, 208; green, 2; blue, 27 }  ,fill opacity=1 ] (458.38,100.33) .. controls (457.66,99.92) and (457.41,99) .. (457.83,98.28) .. controls (458.24,97.57) and (459.16,97.32) .. (459.88,97.74) .. controls (460.59,98.15) and (460.84,99.07) .. (460.43,99.78) .. controls (460.01,100.5) and (459.09,100.75) .. (458.38,100.33) -- cycle ;

%Shape: Polygon Curved [id:ds6253285912604949] 
\draw  [dash pattern={on 0.84pt off 2.51pt}] (460.44,77.94) .. controls (480.44,67.94) and (502.94,72.11) .. (524.44,80.61) .. controls (545.94,89.11) and (601.17,192.42) .. (593.17,203.92) .. controls (585.17,215.42) and (584.17,224.42) .. (501.17,222.42) .. controls (418.17,220.42) and (414.67,218.42) .. (396.17,201.42) .. controls (377.67,184.42) and (440.44,87.94) .. (460.44,77.94) -- cycle ;
%Straight Lines [id:da1189061895885336] 
\draw    (608.17,238.42) -- (544.01,182.65) ;
\draw [shift={(542.5,181.33)}, rotate = 41] [color={rgb, 255:red, 0; green, 0; blue, 0 }  ][line width=0.75]    (10.93,-3.29) .. controls (6.95,-1.4) and (3.31,-0.3) .. (0,0) .. controls (3.31,0.3) and (6.95,1.4) .. (10.93,3.29)   ;
%Shape: Regular Polygon [id:dp7603374031292899] 
\draw  [fill={rgb, 255:red, 155; green, 155; blue, 155 }  ,fill opacity=0.6 ] (196.75,237.64) .. controls (231.82,243.11) and (226.84,268.8) .. (196.95,271.88) .. controls (167.07,274.96) and (113.64,275.31) .. (67.55,271.69) .. controls (21.47,268.06) and (33.46,236.86) .. (69.52,235.91) .. controls (105.59,234.96) and (161.69,232.16) .. (196.75,237.64) -- cycle ;
%Shape: Regular Polygon [id:dp6014039858797512] 
\draw  [fill={rgb, 255:red, 155; green, 155; blue, 155 }  ,fill opacity=0.6 ] (114.11,136.25) .. controls (137.13,112.26) and (158.52,129.47) .. (146.12,153.96) .. controls (133.73,178.45) and (106.76,219.54) .. (79.82,252.94) .. controls (52.89,286.33) and (29.95,260.88) .. (47.5,232.77) .. controls (65.04,204.66) and (91.09,160.25) .. (114.11,136.25) -- cycle ;
%Shape: Polygon Curved [id:ds4194032605402256] 
\draw  [fill={rgb, 255:red, 155; green, 155; blue, 155 }  ,fill opacity=0.6 ] (114.21,146.17) .. controls (120.22,114.05) and (148.27,118.64) .. (151.61,146.02) .. controls (154.95,173.39) and (157.29,204.89) .. (153.28,247.1) .. controls (149.27,289.3) and (114.21,280.13) .. (113.21,247.1) .. controls (112.21,214.07) and (108.2,178.28) .. (114.21,146.17) -- cycle ;
%Shape: Regular Polygon [id:dp3285801982403188] 
\draw  [fill={rgb, 255:red, 155; green, 155; blue, 155 }  ,fill opacity=0.6 ] (118.47,152.99) .. controls (106.86,123.31) and (133.62,114.38) .. (150.87,135.83) .. controls (168.12,157.27) and (194.06,198.18) .. (212.71,235.42) .. controls (231.36,272.65) and (196.02,278.93) .. (177.83,251.67) .. controls (159.64,224.41) and (130.09,182.67) .. (118.47,152.99) -- cycle ;
%Shape: Ellipse [id:dp8997586940701908] 
\draw  [fill={rgb, 255:red, 0; green, 0; blue, 0 }  ,fill opacity=1 ] (90.66,199.71) .. controls (90.66,199.03) and (91.26,198.49) .. (92,198.49) .. controls (92.74,198.49) and (93.34,199.03) .. (93.34,199.71) .. controls (93.34,200.38) and (92.74,200.93) .. (92,200.93) .. controls (91.26,200.93) and (90.66,200.38) .. (90.66,199.71) -- cycle ;
%Shape: Ellipse [id:dp28921147805697767] 
\draw  [fill={rgb, 255:red, 0; green, 0; blue, 0 }  ,fill opacity=1 ] (132.56,141.84) .. controls (132.56,141.16) and (133.15,140.62) .. (133.89,140.62) .. controls (134.63,140.62) and (135.23,141.16) .. (135.23,141.84) .. controls (135.23,142.52) and (134.63,143.06) .. (133.89,143.06) .. controls (133.15,143.06) and (132.56,142.52) .. (132.56,141.84) -- cycle ;
%Shape: Ellipse [id:dp6638098162292292] 
\draw  [fill={rgb, 255:red, 0; green, 0; blue, 0 }  ,fill opacity=1 ] (169.62,197.81) .. controls (169.62,197.13) and (170.22,196.59) .. (170.96,196.59) .. controls (171.7,196.59) and (172.3,197.13) .. (172.3,197.81) .. controls (172.3,198.48) and (171.7,199.03) .. (170.96,199.03) .. controls (170.22,199.03) and (169.62,198.48) .. (169.62,197.81) -- cycle ;
%Shape: Ellipse [id:dp8147452395305121] 
\draw  [fill={rgb, 255:red, 0; green, 0; blue, 0 }  ,fill opacity=1 ] (131.55,215.24) .. controls (131.55,214.57) and (132.15,214.02) .. (132.89,214.02) .. controls (133.63,214.02) and (134.23,214.57) .. (134.23,215.24) .. controls (134.23,215.92) and (133.63,216.47) .. (132.89,216.47) .. controls (132.15,216.47) and (131.55,215.92) .. (131.55,215.24) -- cycle ;
%Shape: Ellipse [id:dp6965898963099839] 
\draw  [fill={rgb, 255:red, 0; green, 0; blue, 0 }  ,fill opacity=1 ] (56.59,251.1) .. controls (56.59,250.42) and (57.19,249.88) .. (57.92,249.88) .. controls (58.66,249.88) and (59.26,250.42) .. (59.26,251.1) .. controls (59.26,251.78) and (58.66,252.32) .. (57.92,252.32) .. controls (57.19,252.32) and (56.59,251.78) .. (56.59,251.1) -- cycle ;
%Shape: Ellipse [id:dp6095499705062852] 
\draw  [fill={rgb, 255:red, 0; green, 0; blue, 0 }  ,fill opacity=1 ] (131.55,259.28) .. controls (131.55,258.61) and (132.15,258.06) .. (132.89,258.06) .. controls (133.63,258.06) and (134.23,258.61) .. (134.23,259.28) .. controls (134.23,259.96) and (133.63,260.51) .. (132.89,260.51) .. controls (132.15,260.51) and (131.55,259.96) .. (131.55,259.28) -- cycle ;
%Shape: Ellipse [id:dp9560165391180352] 
\draw  [fill={rgb, 255:red, 0; green, 0; blue, 0 }  ,fill opacity=1 ] (200.68,253.78) .. controls (200.68,253.1) and (201.28,252.55) .. (202.02,252.55) .. controls (202.75,252.55) and (203.35,253.1) .. (203.35,253.78) .. controls (203.35,254.45) and (202.75,255) .. (202.02,255) .. controls (201.28,255) and (200.68,254.45) .. (200.68,253.78) -- cycle ;

%Shape: Regular Polygon [id:dp7507262951813153] 
\draw  [fill={rgb, 255:red, 155; green, 155; blue, 155 }  ,fill opacity=0.6 ] (443.48,217.12) .. controls (418.73,213.72) and (421.99,196.09) .. (443.02,193.69) .. controls (464.04,191.29) and (501.68,190.53) .. (534.18,192.56) .. controls (566.68,194.59) and (558.54,216.06) .. (533.14,217.06) .. controls (507.74,218.06) and (468.24,220.52) .. (443.48,217.12) -- cycle ;
%Shape: Regular Polygon [id:dp267754919857075] 
\draw  [fill={rgb, 255:red, 155; green, 155; blue, 155 }  ,fill opacity=0.6 ] (502.67,285.64) .. controls (486.69,302.27) and (471.45,290.71) .. (479.95,273.84) .. controls (488.45,256.97) and (507.06,228.61) .. (525.71,205.5) .. controls (544.37,182.4) and (560.77,199.58) .. (548.68,218.98) .. controls (536.59,238.37) and (518.66,269) .. (502.67,285.64) -- cycle ;
%Shape: Polygon Curved [id:ds7824100573485219] 
\draw  [fill={rgb, 255:red, 155; green, 155; blue, 155 }  ,fill opacity=0.6 ] (502.51,278.86) .. controls (498.58,300.88) and (478.78,298.02) .. (476.16,279.33) .. controls (473.55,260.64) and (471.6,239.12) .. (474.02,210.22) .. controls (476.44,181.32) and (501.23,187.25) .. (502.25,209.83) .. controls (503.27,232.41) and (506.44,256.84) .. (502.51,278.86) -- cycle ;
%Shape: Regular Polygon [id:dp6841341656199212] 
\draw  [fill={rgb, 255:red, 155; green, 155; blue, 155 }  ,fill opacity=0.6 ] (499.44,274.23) .. controls (507.91,294.42) and (489.14,300.78) .. (476.78,286.29) .. controls (464.42,271.79) and (445.76,244.07) .. (432.26,218.79) .. controls (418.77,193.51) and (443.61,188.87) .. (456.69,207.34) .. controls (469.76,225.8) and (490.97,254.05) .. (499.44,274.23) -- cycle ;
%Shape: Ellipse [id:dp9418829634230012] 
\draw  [fill={rgb, 255:red, 0; green, 0; blue, 0 }  ,fill opacity=1 ] (518.59,242.01) .. controls (518.59,242.47) and (518.18,242.85) .. (517.66,242.86) .. controls (517.14,242.87) and (516.71,242.5) .. (516.7,242.04) .. controls (516.7,241.58) and (517.11,241.19) .. (517.63,241.19) .. controls (518.15,241.18) and (518.58,241.55) .. (518.59,242.01) -- cycle ;
%Shape: Ellipse [id:dp358088941109576] 
\draw  [fill={rgb, 255:red, 0; green, 0; blue, 0 }  ,fill opacity=1 ] (489.63,282) .. controls (489.63,282.46) and (489.22,282.84) .. (488.7,282.85) .. controls (488.18,282.85) and (487.75,282.49) .. (487.74,282.02) .. controls (487.74,281.56) and (488.15,281.18) .. (488.67,281.17) .. controls (489.19,281.17) and (489.62,281.54) .. (489.63,282) -- cycle ;
%Shape: Ellipse [id:dp720690761673329] 
\draw  [fill={rgb, 255:red, 0; green, 0; blue, 0 }  ,fill opacity=1 ] (462.98,244.09) .. controls (462.99,244.55) and (462.57,244.93) .. (462.05,244.94) .. controls (461.53,244.94) and (461.1,244.58) .. (461.1,244.11) .. controls (461.09,243.65) and (461.51,243.27) .. (462.03,243.26) .. controls (462.55,243.26) and (462.97,243.63) .. (462.98,244.09) -- cycle ;
%Shape: Ellipse [id:dp38845455990424327] 
\draw  [fill={rgb, 255:red, 0; green, 0; blue, 0 }  ,fill opacity=1 ] (489.63,231.79) .. controls (489.64,232.25) and (489.22,232.63) .. (488.7,232.64) .. controls (488.18,232.65) and (487.76,232.28) .. (487.75,231.82) .. controls (487.74,231.36) and (488.16,230.97) .. (488.68,230.97) .. controls (489.2,230.96) and (489.62,231.33) .. (489.63,231.79) -- cycle ;
%Shape: Ellipse [id:dp479815058477143] 
\draw  [color={rgb, 255:red, 248; green, 231; blue, 28 }  ,draw opacity=1 ][fill={rgb, 255:red, 248; green, 231; blue, 28 }  ,fill opacity=1 ] (542.09,205.87) .. controls (542.1,206.7) and (541.44,207.38) .. (540.61,207.39) .. controls (539.79,207.4) and (539.11,206.74) .. (539.09,205.91) .. controls (539.08,205.08) and (539.74,204.4) .. (540.57,204.39) .. controls (541.4,204.38) and (542.08,205.04) .. (542.09,205.87) -- cycle ;
%Shape: Ellipse [id:dp08119213325969932] 
\draw  [color={rgb, 255:red, 144; green, 19; blue, 254 }  ,draw opacity=1 ][fill={rgb, 255:red, 144; green, 19; blue, 254 }  ,fill opacity=1 ] (489.37,203.24) .. controls (489.38,204.07) and (488.72,204.75) .. (487.89,204.76) .. controls (487.06,204.77) and (486.38,204.11) .. (486.37,203.28) .. controls (486.36,202.45) and (487.02,201.77) .. (487.85,201.76) .. controls (488.68,201.75) and (489.36,202.41) .. (489.37,203.24) -- cycle ;
%Shape: Ellipse [id:dp011379514337418462] 
\draw  [color={rgb, 255:red, 208; green, 2; blue, 27 }  ,draw opacity=1 ][fill={rgb, 255:red, 208; green, 2; blue, 27 }  ,fill opacity=1 ] (440.56,205.45) .. controls (440.57,206.28) and (439.91,206.96) .. (439.08,206.98) .. controls (438.25,206.99) and (437.57,206.32) .. (437.56,205.5) .. controls (437.54,204.67) and (438.21,203.99) .. (439.04,203.98) .. controls (439.86,203.96) and (440.54,204.63) .. (440.56,205.45) -- cycle ;

%Shape: Regular Polygon [id:dp37663219459452624] 
\draw  [fill={rgb, 255:red, 155; green, 155; blue, 155 }  ,fill opacity=0.6 ] (534.19,98.88) .. controls (518.58,79.36) and (505.12,91.19) .. (513.83,110.48) .. controls (522.54,129.76) and (541.17,162.48) .. (559.57,189.35) .. controls (577.97,216.22) and (592.23,198.22) .. (580.07,175.9) .. controls (567.92,153.57) and (549.8,118.4) .. (534.19,98.88) -- cycle ;
%Shape: Regular Polygon [id:dp8145724455839541] 
\draw  [fill={rgb, 255:red, 155; green, 155; blue, 155 }  ,fill opacity=0.6 ] (623.36,114.6) .. controls (629.45,92.35) and (611.72,85.19) .. (601.59,101.13) .. controls (591.45,117.08) and (576.63,147.59) .. (566.35,175.44) .. controls (556.06,203.3) and (579.22,208.59) .. (589.68,188.27) .. controls (600.14,167.95) and (617.26,136.86) .. (623.36,114.6) -- cycle ;
%Shape: Polygon Curved [id:ds47405266424412695] 
\draw  [fill={rgb, 255:red, 155; green, 155; blue, 155 }  ,fill opacity=0.6 ] (617.45,117.94) .. controls (634.35,103.28) and (621.74,87.74) .. (604.36,95.07) .. controls (586.97,102.4) and (567.49,111.75) .. (543.91,128.63) .. controls (520.33,145.52) and (538.12,163.77) .. (558.04,153.08) .. controls (577.95,142.39) and (600.55,132.59) .. (617.45,117.94) -- cycle ;
%Shape: Regular Polygon [id:dp46808532578652573] 
\draw  [fill={rgb, 255:red, 155; green, 155; blue, 155 }  ,fill opacity=0.6 ] (611.91,117.67) .. controls (633.58,114.6) and (629.43,95.22) .. (610.65,92.03) .. controls (591.87,88.85) and (558.5,87.02) .. (529.88,88.38) .. controls (501.26,89.75) and (510,113.46) .. (532.55,115.23) .. controls (555.1,117) and (590.24,120.74) .. (611.91,117.67) -- cycle ;
%Shape: Ellipse [id:dp7673552380811842] 
\draw  [fill={rgb, 255:red, 0; green, 0; blue, 0 }  ,fill opacity=1 ] (594.05,150.62) .. controls (594.45,150.39) and (594.56,149.84) .. (594.3,149.39) .. controls (594.04,148.94) and (593.5,148.76) .. (593.1,148.99) .. controls (592.7,149.22) and (592.59,149.77) .. (592.85,150.22) .. controls (593.11,150.67) and (593.65,150.85) .. (594.05,150.62) -- cycle ;
%Shape: Ellipse [id:dp22505463054403874] 
\draw  [fill={rgb, 255:red, 0; green, 0; blue, 0 }  ,fill opacity=1 ] (613.55,105.26) .. controls (613.95,105.03) and (614.06,104.48) .. (613.8,104.03) .. controls (613.54,103.58) and (613.01,103.4) .. (612.61,103.63) .. controls (612.21,103.86) and (612.09,104.42) .. (612.35,104.87) .. controls (612.61,105.32) and (613.15,105.49) .. (613.55,105.26) -- cycle ;
%Shape: Ellipse [id:dp3583035570013161] 
\draw  [fill={rgb, 255:red, 0; green, 0; blue, 0 }  ,fill opacity=1 ] (567.34,101.8) .. controls (567.74,101.57) and (567.85,101.02) .. (567.59,100.57) .. controls (567.33,100.12) and (566.8,99.94) .. (566.4,100.17) .. controls (566,100.4) and (565.88,100.95) .. (566.14,101.4) .. controls (566.4,101.85) and (566.94,102.03) .. (567.34,101.8) -- cycle ;
%Shape: Ellipse [id:dp5394991450015008] 
\draw  [fill={rgb, 255:red, 0; green, 0; blue, 0 }  ,fill opacity=1 ] (570.43,130.99) .. controls (570.83,130.76) and (570.95,130.21) .. (570.69,129.76) .. controls (570.43,129.31) and (569.89,129.13) .. (569.49,129.36) .. controls (569.09,129.59) and (568.98,130.14) .. (569.24,130.59) .. controls (569.5,131.04) and (570.03,131.22) .. (570.43,130.99) -- cycle ;
%Shape: Ellipse [id:dp042917742916345314] 
\draw  [color={rgb, 255:red, 248; green, 231; blue, 28 }  ,draw opacity=1 ][fill={rgb, 255:red, 248; green, 231; blue, 28 }  ,fill opacity=1 ] (575.05,189.32) .. controls (575.77,188.91) and (576.01,187.99) .. (575.6,187.27) .. controls (575.18,186.56) and (574.27,186.31) .. (573.55,186.73) .. controls (572.83,187.14) and (572.59,188.06) .. (573,188.78) .. controls (573.42,189.49) and (574.33,189.74) .. (575.05,189.32) -- cycle ;
%Shape: Ellipse [id:dp1848513230031411] 
\draw  [color={rgb, 255:red, 144; green, 19; blue, 254 }  ,draw opacity=1 ][fill={rgb, 255:red, 144; green, 19; blue, 254 }  ,fill opacity=1 ] (545.78,145.39) .. controls (546.5,144.98) and (546.74,144.06) .. (546.33,143.34) .. controls (545.91,142.62) and (545,142.38) .. (544.28,142.79) .. controls (543.56,143.21) and (543.32,144.13) .. (543.73,144.84) .. controls (544.14,145.56) and (545.06,145.81) .. (545.78,145.39) -- cycle ;
%Shape: Ellipse [id:dp6938809681373537] 
\draw  [color={rgb, 255:red, 208; green, 2; blue, 27 }  ,draw opacity=1 ][fill={rgb, 255:red, 208; green, 2; blue, 27 }  ,fill opacity=1 ] (522.67,102.34) .. controls (523.39,101.92) and (523.64,101.01) .. (523.22,100.29) .. controls (522.81,99.57) and (521.89,99.33) .. (521.17,99.74) .. controls (520.45,100.15) and (520.21,101.07) .. (520.62,101.79) .. controls (521.04,102.51) and (521.96,102.75) .. (522.67,102.34) -- cycle ;

% Text Node
\draw (614.5,239.5) node [anchor=north west][inner sep=0.75pt]   [align=left] {$K^{(3)}_{3,3,3}$};

\end{tikzpicture}
    \caption{The fan $\mathcal{F}_3$ and the semi-blowup fan $\mathcal{F}_{3,3}$.
In the drawing of $\mathcal{F}_{3,3}$, the dashed region denotes the complete tripartite $3$-graph $K_{3,3,3}^{(3)}$ on the transversal vertices, whose three partite classes are indicated by the three colors.}
    \label{fig:fan}
\end{figure}

We first recall the definition of the fan that appears in the Mubayi--Pikhurko generalization of
Mantel's theorem to hypergraphs~\cite{MubayiPikhurko2007}.

\begin{defn}\label{def:fan}
For $r\ge 2$, the \emph{fan} $\mathcal{F}_r$ is the $r$-graph with vertex set
\[
        \{c,t_1,\dots,t_r\}\cup A_1\cup\cdots\cup A_r,
\]
where $A_1, \ldots, A_r$ are pairwise disjoint sets of vertices, each of size $r-2$. Its edge set consists of
\[
        T:=\{t_1,\dots,t_r\}
        \quad\text{and}\quad
        P_i:=\{c,t_i\}\cup A_i \quad\text{for } i\in [r].
\]
We call $T$ the \emph{transversal edge} and $P_1,\dots,P_r$ the
\emph{petals}.
\end{defn}

% Note that for $r=2$, this fan is a triangle.

Our forbidden configuration is obtained by taking several disjoint fans and
then adding the natural transversal edges between their transversal vertices (see Figure~\ref{fig:fan}).

\begin{defn}[Semi-blowup fan]\label{def:Frt}
Fix $r\ge 2$ and $t\ge r$. Let
$\mathcal{F}_r^1,\dots,\mathcal{F}_r^t$ be $t$ vertex-disjoint labelled copies of $\mathcal{F}_r$. In the
$i$th copy, write $c^i$, $t_1^i,\ldots,t_r^i$, and $A_1^i, \ldots,A_r^i$ for the \emph{center}, \emph{transversal vertices}, and \emph{petal vertex sets}, respectively.
The \emph{semi-blowup fan} $\mathcal{F}_{r,t}$ is the $r$-graph on the union of these copies whose edges are
\begin{enumerate}[label=(\roman*)]
\item all petals and all transversal edges of the individual copies $\mathcal{F}_r^i$, $i\in[t]$;
\item all \emph{blowup transversal edges} $\{t_1^{i_1},t_2^{i_2},\dots,t_r^{i_r}\}$ for $(i_1,\dots,i_r)\in [t]^r$ with $|\{i_1,\ldots,i_r\}|\ge 2$.
\end{enumerate}
\end{defn}

Equivalently, $\mathcal{F}_{r,t}$ is obtained from the full blowup $\mathcal{F}_r[t]$ by keeping the complete blowup of the transversal edge and, for each copy, keeping only its original petals.  The $r$-graph $\mathcal F_{r,t}$ is not $r$-partite: already $\mathcal F_r$ is not $r$-partite, since the transversal edge would force $t_1,\ldots,t_r$ into distinct parts, leaving no possible part for the center $c$ in all petals.

\begin{remark}
For each $j\in[r]$, let $T_j:=\{t_j^i:i\in[t]\}$. 
Note that the induced subgraph of $\mathcal{F}_{r,t}$ on the set $\bigcup_{j \in [r]} T_j$ of all transversal vertices is the complete $r$-partite $r$-graph with parts $T_1,\ldots,T_r$.
Indeed, every edge in Definition~\ref{def:Frt}(ii) contains exactly one vertex from each of the parts $T_1,\ldots,T_r$. Moreover, these edges form the complete $r$-partite $r$-graph on these parts with the $t$ diagonal edges $\{t_1^i,\ldots,t_r^i\}$, $i\in[t]$, removed. These diagonal edges are precisely the individual transversal edges already included in Definition~\ref{def:Frt}(i). 
\end{remark}

\section{Constructions with fewer copies of $\mathcal{F}_{r,t}$}\label{sec:ss}

This section proves Theorem~\ref{thm:main}, assuming the exact result
Theorem~\ref{thm:exact}, whose proof is postponed to Section~\ref{sec:exact}.

% Together with Lemma~\ref{thm:weak-stability}, Theorem~\ref{thm:exact} also gives stability. Indeed, since $|\mathcal{J}_{t-1}(n)|=t_r^{(r)}(n)+O_{r,t}(n^{r-1})$, every $\mathcal{F}_{r,t}$-free $r$-graph with $\ex(n,\mathcal{F}_{r,t})-o(n^r)$ edges is $o(n^r)$-close to $\mathcal{J}_{t-1}(n)$. Therefore, $\mathcal{F}_{r,t}$ is stable by definition.

Fix integers $r\ge 2$ and $t\ge r$, and set
\[
        \mathcal F:=\mathcal F_{r,t}
        \quad\text{and}\quad 
        \mathcal J:=\mathcal J_{t-1}(n).
\]

Let $S$ and $V_1\cup\cdots\cup V_r$ be the exceptional set and the balanced
partition of the partite vertices of $\mathcal J$, as defined in
Section~\ref{sec:prelim}. Thus $\mathcal J-S=\mathcal T_r^{(r)}(n-t+1)$. We write
\[
        \overline{\mathcal J}_{\mathrm{par}}
        :=
        \binom{V(\mathcal J)\setminus S}{r}\setminus \mathcal J
\]
for the family of partite non-edges of $\mathcal J$; equivalently, these are
the non-crossing $r$-sets in $V(\mathcal J)\setminus S$. Put
$$
        m_{\mathcal F}:=v(\mathcal F)-r-(t-1).
$$
For $Z\subseteq V(\mathcal J)\setminus S$, define
$$
        \mathcal S(Z):=\left\{\{z\}\cup Y:
        z\in Z \text{ and } Y\in \tbinom{S}{r-1}\right\}.
$$
We call $\mathcal S(Z)$ the \emph{$S$-star of $Z$}. Equivalently,
$\mathcal S(Z)$ consists of the edges of $\mathcal J$
which contain exactly one vertex of $Z$ and otherwise lie in $S$. Put
$$
        m_S:=\binom{|S|}{r-1}=\binom{t-1}{r-1},
$$
so that $|\mathcal S(Z)|=m_S|Z|$.
% All implicit constants in this section may depend on the fixed
% $r$-graph $\mathcal F$.

The construction starts from the extremal $r$-graph
$\mathcal J=\mathcal J_{t-1}(n)$.  For a suitable small set $Z$ of partite
vertices, we delete the $S$-star $\mathcal S(Z)$ and replace the deleted edges
by partite non-edges contained in $Z$.  The purpose of the deletion is local:
it destroys almost all copies created by any one added non-edge, while the
copies using at least two added non-edges are of lower order in the range
$q\le \delta n$.

The required estimates are organized as follows. Lemmas~\ref{lem:one-edge-active}
and~\ref{lem:c-scale} describe the copies created by one added member of
$\overline{\mathcal J}_{\mathrm{par}}$ and give the scale
$c(n,\mathcal F)=\Theta_{\mathcal F}(n^{m_{\mathcal F}})$. Lemmas~\ref{lem:single-fan}
and~\ref{lem:star-deletion} prove that deleting $\mathcal S(Z)$ leaves only a
$\rho_t$-proportion of the relevant one-new-edge copies, with
$\rho_t\to0$ as $t\to\infty$. Lemma~\ref{lem:two-new} bounds copies using
at least two added edges. Finally, Lemma~\ref{lem:selection} supplies enough
members of $\overline{\mathcal J}_{\mathrm{par}}$ of the required part-type,
and these estimates
are combined in the proof of Theorem~\ref{thm:main}.

\subsection{Copies with added edges}

We first understand the local effect of adding one member of
$\overline{\mathcal J}_{\mathrm{par}}$ to
$\mathcal J$. The main point is that every copy created by this edge has a
distinguished active fan, while the other $t-1$ fans are forced to use the
exceptional set $S$.

\begin{lemma}\label{lem:one-edge-active}
Let $e\in \overline{\mathcal J}_{\mathrm{par}}$, and let $\varphi$ be an
embedding of $\mathcal F$ into $\mathcal J\cup\{e\}$ such that
$e\in\varphi(\mathcal F)$. Then there is a unique $i\in[t]$ and a unique
edge $f\in\mathcal F_r^i$ such that $\varphi(f)=e$. Moreover,
\[
        |\varphi(V(\mathcal F_r^j))\cap S|=1
        \quad\text{for every } j\in[t]\setminus\{i\},
\]
and consequently,
\[
        S=\bigcup_{j\in[t]\setminus\{i\}}
        \bigl(\varphi(V(\mathcal F_r^j))\cap S\bigr).
\]
\end{lemma}

\begin{proof}
Suppose first that no edge of an individual fan is mapped to $e$. Then $e$ is
the image of a blowup transversal edge involving vertices from at least two
different fans. Hence all petals and the transversal edge of each of the $t$
fans are mapped into $\mathcal J$. Since $\mathcal J-S$ is $r$-partite and
$\mathcal F_r$ is not $r$-partite, each of these $t$ vertex-disjoint fans must
meet $S$, which is impossible because $|S|=t-1$.
Thus some edge of an individual fan is mapped to $e$. This fan is unique; write
it as $\mathcal F_r^i$. The edge $f\in\mathcal F_r^i$ with $\varphi(f)=e$ is
also unique, since $\varphi$ is injective on vertices.

For every $j\ne i$, all edges of $\mathcal F_r^j$ are mapped into
$\mathcal J$, so $\varphi(V(\mathcal F_r^j))$ must meet $S$. The sets
$\varphi(V(\mathcal F_r^j))\cap S$, $j\ne i$, are pairwise disjoint and
nonempty, and there are $t-1$ of them inside the set $S$ of size $t-1$.
Therefore each has size one and their union is $S$.
\end{proof}

For the rest of this section, in any embedding counted by
Lemma~\ref{lem:one-edge-active}, the unique fan whose edge is mapped to the
added edge $e$ is called the \emph{active fan}, and the other fans are called
\emph{inactive fans}.

The next lemma turns this structural description into the scale of the
one-edge supersaturation function.

\begin{lemma}\label{lem:c-scale}
For every $e\in \overline{\mathcal J}_{\mathrm{par}}$, $N_{\mathcal{F}}(\mathcal{J}\cup \{e\};e)=\Theta_{\mathcal{F}}(n^{m_{\mathcal F}})$.
Consequently, $c(n,\mathcal{F})=\Theta_{\mathcal{F}}(n^{m_{\mathcal F}})$.
\end{lemma}

\begin{proof}
By Lemma~\ref{lem:one-edge-active}, any
copy of $\mathcal{F}$ in $\mathcal{J}\cup \{e\}$ containing $e$ contains the fixed set $e\cup S$, of size
$r+t-1$. After choosing the edge of $\mathcal{F}$ mapped to $e$ and the vertices of $\mathcal{F}$
mapped to $S$, there are only $O_{\mathcal{F}}(1)$ possible finite types. The remaining vertices are chosen from $V(\mathcal{J})\setminus S$, giving at most $O_{\mathcal{F}}\bigl(n^{v(\mathcal{F})-r-(t-1)}\bigr)=O_{\mathcal{F}}(n^{m_{\mathcal F}})$ copies.
This gives the upper bound.

For the lower bound, fix $e\in \overline{\mathcal J}_{\mathrm{par}}$.
Since $e$ is non-crossing, choose $i\in[r]$ and distinct vertices
$u,v\in e\cap V_i$.  In the first fan, map the petal $P_i$ to $e$, with the
center mapped to $u$ and $t_i^1$ mapped to $v$.  Choose the remaining
transversal vertices of this fan by putting $t_j^1\in V_j$ for $j\ne i$, and
choose the remaining vertices of the other petals so that those petals and the
transversal edge of the first fan are crossing.

For each of the other $t-1$ fans, assign a distinct vertex of $S$ to one
transversal coordinate.  If $s\in S$ is assigned to coordinate $\ell$, map
$t_\ell$ to $s$, put the center in $V_\ell$, put $t_j$ in $V_j$ for
$j\ne \ell$, and choose all remaining petal vertices from fixed prescribed
parts, avoiding the vertices already used.  Then every petal and every
transversal edge of this inactive fan is present: the edges containing $s$
meet $S$, and the others are crossing.  Moreover, every blowup transversal
edge is either crossing or meets $S$.  Thus these choices give labelled copies
of $\mathcal F$ containing $e$.  The only vertices fixed in advance are those
of $e\cup S$, and all remaining choices are made from linear-sized parts while
avoiding only $O_{\mathcal F}(1)$ vertices.  Hence there are
$\Omega_{\mathcal F}(n^{v(\mathcal F)-r-(t-1)})
=\Omega_{\mathcal F}(n^{m_{\mathcal F}})$ such copies.
\end{proof}

After the one-new-edge copies have been controlled, we also need to bound
copies using two or more added edges. This contribution is smaller by one
power of $n$, and hence can be absorbed once the range $q\le \delta n$ is made
sufficiently small.

\begin{lemma}\label{lem:two-new}
Let $E^+\subseteq \overline{\mathcal J}_{\mathrm{par}}$ with
$|E^+|=m_0=O_{\mathcal{F}}(n)$, and let
$\mathcal{H}\subseteq \mathcal{J}\cup E^+$. Then the number of copies of $\mathcal{F}$ in $\mathcal{H}$ that use at least two edges of $E^+$ is at most $a_{\mathcal{F}}m_0^2n^{m_{\mathcal F}-1}$.
\end{lemma}

\begin{proof}
Let $\mathcal{Q}$ be such a $\mathcal{F}$ copy. Since
$E^+\subseteq \overline{\mathcal J}_{\mathrm{par}}$, at least one individual fan has
a petal or its own transversal edge in $E^+$. Indeed, if this were not the
case, then all petals and all transversal edges of the $t$ individual fans
would lie in $\mathcal{J}$. Since $\mathcal{J}-S$ is $r$-partite whereas
$\mathcal F_r$ is not $r$-partite, each individual fan would have to meet
$S$. The $t$ individual fans are vertex-disjoint, so this would require $t$
distinct vertices of $S$, impossible because $|S|=t-1$.

Let $k$ be the number of individual fans having a petal or its own
transversal edge in $E^+$. If $k=1$, then the other $t-1$ fans use all
vertices of $S$. Choose one such new edge in the unique active fan and one
further new edge of $\mathcal{Q}$. These two members of
$\overline{\mathcal J}_{\mathrm{par}}$ are distinct,
so their union has size at least $r+1$ and is disjoint from $S$. Thus at least
$r+t$ vertices of $\mathcal{Q}$ are fixed, and at most
\[
        v(\mathcal{F})-(r+t)=m_{\mathcal F}-1
\]
vertices remain free. There are $O_{\mathcal{F}}(m_0^2)$ choices for the two distinguished
new edges and $O_{\mathcal{F}}(1)$ choices for their roles in $\mathcal{F}$, giving
$O_{\mathcal{F}}(m_0^2n^{m_{\mathcal F}-1})$ copies.
If $k\ge 2$, choose one such new edge from each of these $k$
fans and ignore any additional new edges of $\mathcal{Q}$.
Then the chosen edges lie in disjoint fans and fix $kr$ partite vertices.
The remaining $t-k$ fans have all their petals and their transversal edges in $\mathcal{J}$, and hence each contains a distinct vertex of $S$.
Thus at least $kr+(t-k)=t+k(r-1)$ vertices are fixed.
For this value of $k$, the number of copies is at most
        $O_{\mathcal F}\bigl(m_0^k n^{v(\mathcal F)-t-k(r-1)}\bigr).$
Since $m_0=O_{\mathcal F}(n)$ and $k\ge2$,
\[
        m_0^k n^{v(\mathcal F)-t-k(r-1)}
        \le
        O_{\mathcal F}\bigl(m_0^2
        n^{v(\mathcal F)-t-k(r-2)-2}\bigr)
        \le
        O_{\mathcal F}\bigl(m_0^2 n^{v(\mathcal F)-r-t}\bigr),
\]
where the last inequality uses $k(r-2)+2\ge r$.  This is
$O_{\mathcal F}(m_0^2n^{m_{\mathcal F}-1})$.  Summing over $k$ proves the lemma.
\end{proof}

\subsection{Deleting the \texorpdfstring{$S$}{S}-star}

We next show why deleting the $S$-star of a small set $Z$ is effective.
Among the one-new-edge copies counted above, only an exponentially small
proportion can avoid the deleted $S$-star when $t$ is large.

Recall that in $\mathcal F_r$ the vertices $t_1,\ldots,t_r$ are the
transversal vertices. In the next lemma, a permutation $\sigma$ prescribes
the part of each transversal coordinate: the $j$th transversal
vertex is required to lie in $V_{\sigma(j)}$, unless it is the unique
exceptional vertex $s$.

\begin{lemma}\label{lem:single-fan}
Fix $s\in S$, a bounded set $R$ of partite vertices, and a
permutation $\sigma$ of $[r]$. Let $\mathscr A$ be the set of labelled edge-preserving embeddings $\varphi:\mathcal F_r\to \mathcal J-R$
such that $\varphi(V(\mathcal F_r))\cap S=\{s\}$
and, for every $j\in[r]$,
\[
\varphi(t_j)=s \quad\text{or}\quad \varphi(t_j)\in V_{\sigma(j)}.
\]
Then there is a constant $\vartheta_r>0$, depending only on $r$, such
that for every $j\in[r]$,
\[
|\{\varphi\in\mathscr A:\varphi(t_j)=s\}|
   \ge \vartheta_r|\mathscr A|.
\]
\end{lemma}

\begin{proof}
Let $h=v(\mathcal F_r)$. Since $s$ is the only exceptional vertex used by such
an embedding, the vertex of $\mathcal F_r$ mapped to $s$ has at most $h$
choices. Once this vertex and the part $V_i$ containing the image of each
remaining vertex are fixed, there are at most $O_r(n^{h-1})$ labelled embeddings. Hence $|\mathscr A|\le b_1(r)n^{h-1}$ for some constant $b_1(r)$.

Fix $j\in[r]$. We construct many embeddings with $\varphi(t_j)=s$. Put the
center in $V_{\sigma(j)}$, and put $t_\ell$ in $V_{\sigma(\ell)}$ for every
$\ell\ne j$. For each $\ell\ne j$, choose the remaining vertices of the petal
$P_\ell$ in the other parts, so that $P_\ell$ is crossing. For the petal
$P_j$, which contains $s$, choose its remaining vertices in one fixed
part, avoiding $R$ and the vertices already chosen. Since every edge meeting
$S$ belongs to $\mathcal J$, this gives at least $c_1(r)n^{h-1}$ embeddings for
all sufficiently large $n$. Taking $\vartheta_r:=c_1(r)/b_1(r)$
proves the lemma.
\end{proof}

We now apply the preceding single-fan estimate to all inactive fans in a
one-new-edge copy. The conclusion is that the $S$-star deletion removes all but
an exponentially small proportion of such copies.

\begin{lemma}\label{lem:star-deletion}
For every fixed $r\ge2$ there are constants $\theta_r>0$ and $c_r>0$ such that
the following holds. Let $t\ge r$, let
$e\in \overline{\mathcal J}_{\mathrm{par}}$, and let
$Z\subseteq V(\mathcal J)\setminus S$ with $e\subseteq Z$.
Then, for all sufficiently large $n$,
\[
        N_{\mathcal F}\bigl((\mathcal J-\mathcal S(Z))\cup\{e\};e\bigr)
        \le
        \rho_t N_{\mathcal F}(\mathcal J\cup\{e\};e),
\]
where $\rho_t:=c_r(1-\theta_r)^{t-1}$.
In particular, $\rho_t\to0$ as $t\to\infty$.
\end{lemma}

\begin{proof}
Let $\vartheta_r$ be the constant from Lemma~\ref{lem:single-fan}, and put
$\theta_r=\vartheta_r/2$, taking $n$ sufficiently large whenever needed. Let
$\mathscr C(e)$ be the family of labelled copies of $\mathcal F$ in
$\mathcal J\cup\{e\}$ which contain $e$. We shall use the following simple
observation throughout: every edge of such a copy contained in
$V(\mathcal J)\setminus S$, except possibly $e$, must be a crossing edge of
$\mathcal J$.

\medskip
\begin{claim}\label{CLAIM-1}
Every copy in $\mathscr C(e)$ which survives in
$(\mathcal J-\mathcal S(Z))\cup\{e\}$ uses $e$ as a petal of its active fan.
\end{claim}
\begin{proof}[Proof of Claim~\ref{CLAIM-1}]
By Lemma~\ref{lem:one-edge-active}, the edge $e$ lies in the active fan, and
the remaining $t-1$ inactive fans use the vertices of $S$ bijectively. Suppose
for a contradiction that $e$ is the transversal edge of the active fan, say $e=\{t_1^i,\ldots,t_r^i\}$ and write $t_k^i\in V_{p_k}$. Since $e\notin\mathcal J$, the sequence
$(p_1,\ldots,p_r)$ is not a permutation of $[r]$.
Call a coordinate $k$ \emph{feasible} if deleting $p_k$ leaves
$r-1$ distinct entries.  Since $(p_1,\ldots,p_r)$ is not a permutation,
the feasible coordinates are very restricted: there are none unless the
multiset $\{p_1,\ldots,p_r\}$ has exactly one repeated part, appearing twice,
and exactly one missing part; in that exceptional case the two feasible
coordinates are precisely the two coordinates carrying the repeated part.

Assume first that $r\ge3$. Let $\mathcal F_r^j$ be an inactive fan, and suppose
that $t_k^j$ is a partite transversal vertex lying in $V_g$.  Replacing
$t_k^i$ by $t_k^j$ in the active transversal edge gives a blowup transversal
edge contained in $V(\mathcal J)\setminus S$ and different from $e$.  Hence
this edge must be crossing, so
\[
        (p_1,\ldots,p_{k-1},g,p_{k+1},\ldots,p_r)
\]
is a permutation of $[r]$.  Thus $k$ is feasible, and $g$ is the unique
missing part after $p_k$ is deleted.  If $r\ge4$, this is impossible because
an inactive fan has at least $r-1\ge3$ partite transversal vertices, but there
are at most two feasible coordinates.

It remains, within this case, to consider $r=3$.  The same argument shows that
an inactive fan cannot have all three transversal vertices outside $S$.
Thus it has exactly two partite transversal vertices.  These must lie in the
two feasible coordinates, and both must lie in the unique missing part of
$(p_1,p_2,p_3)$.  The two feasible active coordinates are the two coordinates
carrying the repeated part.  Taking these two inactive partite transversal
vertices together with the active transversal vertex in the remaining
coordinate gives a non-crossing blowup transversal edge contained in
$V(\mathcal J)\setminus S$, different from $e$.  This is impossible.

It remains to consider $r=2$. Then the two vertices of $e$ lie in the same
part. An inactive triangle cannot have its center in $S$, for then its
two transversal vertices would both be forced into the other part, making
its own transversal edge a member of $\overline{\mathcal J}_{\mathrm{par}}$
different from $e$. Hence every
inactive triangle uses its unique exceptional vertex in a transversal
coordinate. This coordinate is common to all inactive triangles. Otherwise,
two inactive triangles would provide partite transversal vertices in different
coordinates, and these two vertices would form another member of
$\overline{\mathcal J}_{\mathrm{par}}$. Let
the common exceptional coordinate be $\ell$. For any inactive triangle
$\mathcal F_r^j$, the blowup transversal edge $\{t_\ell^j,t_{3-\ell}^i\}$
is mapped to $\{s,z\}$ with $s\in S$ and $z\in e\subseteq Z$, and so lies in
$\mathcal S(Z)$. Thus, the copy is deleted, a contradiction.

Therefore, $e$ is not the transversal edge of the active fan. Since it lies in
that fan, it must be a petal.
\end{proof}

It remains to count copies in which $e$ is a petal of the active fan.

\begin{claim}\label{CLAIM-2}
Let $\mathcal Q\in\mathscr C(e)$ be a copy in which $e$ is a petal of the
active fan $\mathcal F_r^i$. Then there is a permutation $\sigma$ of $[r]$ such
that $t_\ell^i\in V_{\sigma(\ell)}$ for every $\ell\in[r]$ and, for every
inactive fan $\mathcal F_r^j$, every partite transversal vertex $t_\ell^j$
lies in $V_{\sigma(\ell)}$. Conversely, once the individual inactive fans are
embedded in $\mathcal J$ with these transversal-coordinate constraints, all
blowup transversal edges involving at least one inactive fan are present in
$\mathcal J$.
\end{claim}

\begin{proof}[Proof of Claim~\ref{CLAIM-2}]
The inactive fans use all vertices of $S$ bijectively, so the active fan is
disjoint from $S$. Since $e$ is a petal, the transversal edge of the active
fan is an edge of $\mathcal J$ contained in $V(\mathcal J)\setminus S$. Hence it is crossing, and determines
a permutation $\sigma$ with $t_\ell^i\in V_{\sigma(\ell)}$ for all
$\ell\in[r]$.

Let $\mathcal F_r^j$ be an inactive fan, and suppose that $t_\ell^j$ is a partite
vertex. Replacing $t_\ell^i$ in the active transversal edge by $t_\ell^j$ gives a
blowup transversal edge contained in $V(\mathcal J)\setminus S$ and different from $e$. Hence, this edge is
crossing, which forces $t_\ell^j\in V_{\sigma(\ell)}$.
Conversely, consider a blowup transversal edge
$\{t_1^{j_1},\ldots,t_r^{j_r}\}$ involving at least one inactive fan. If it
meets $S$, then it belongs to $\mathcal J$. If it is contained in $V(\mathcal J)\setminus S$, then each of its
vertices lies in a different part of $V_{\sigma(1)},\ldots,V_{\sigma(r)}$, and hence it is a crossing edge of $\mathcal J$.
\end{proof}

We fix the active fan, its embedding, and the assignment of the vertices of
$S$ to the inactive fans. Let $P_k^i$ be the petal of the active fan mapped to
$e$. By Claim~\ref{CLAIM-2}, these choices determine a permutation $\sigma$. The transversal vertex
$t_k^i$ is mapped to a vertex of $e$, and we write this vertex as $z\in Z$.

Let $B$ be the set of indices of the inactive fans. For each $j\in B$, let
$s_j\in S$ be the vertex assigned to $\mathcal F_r^j$. Let $\mathscr C_0$ be
the family of labelled edge-preserving embeddings $\psi$ of the disjoint union
of the inactive fans into $\mathcal J$ such that $\psi$ avoids the fixed active fan, maps each
$\mathcal F_r^j$ using $s_j$ as its unique exceptional vertex, and, for every
$j\in B$ and every $\ell\in[r]$, satisfies
$\psi(t_\ell^j)\in V_{\sigma(\ell)}$ whenever $\psi(t_\ell^j)\notin S$.
By Claim~\ref{CLAIM-2}, every member of
$\mathscr C_0$, together with the fixed active fan, gives a copy of
$\mathcal F$ in $\mathcal J\cup\{e\}$.
If $\varpi$ is a partial embedding of some inactive fans, write
\[
        \mathscr C(\varpi)
        :=
        \{\psi\in\mathscr C_0:\psi\text{ extends }\varpi\}.
\]

\begin{claim}\label{CLAIM-3}
Let $\varpi$ be a partial embedding of some inactive fans with
$\mathscr C(\varpi)\ne\emptyset$. Let $j\in B$ be an index not yet
embedded by $\varpi$. For $\ell\in[r]$, put
\[
        \mathscr C_\ell(\varpi,j)
        :=
        \{\psi\in\mathscr C(\varpi):\psi(t_\ell^j)=s_j\}.
\]
Then
$|\mathscr C_\ell(\varpi,j)|
\ge \theta_r|\mathscr C(\varpi)|$ for every $\ell\in[r]$.
\end{claim}
\begin{proof}[Proof of Claim~\ref{CLAIM-3}]
Let $R$ be the set of partite vertices already unavailable, namely the partite
vertices in the fixed active fan together with those used by $\varpi$.
Let $\mathscr A_j$ be the set of labelled edge-preserving embeddings $\phi$ of
$\mathcal F_r^j$ into $\mathcal J-R$ which use $s_j$ as their unique exceptional vertex
and satisfy $\phi(t_\ell^j)\in V_{\sigma(\ell)}$ whenever
$\phi(t_\ell^j)\ne s_j$.
For $\phi\in \mathscr A_j$, let
\[
        w(\phi)
        :=
        |\{\psi\in\mathscr C(\varpi):
              \psi|_{V(\mathcal F_r^j)}=\phi\}|.
\]
Then
\[
        |\mathscr C(\varpi)|
        =
        \sum_{\phi\in \mathscr A_j} w(\phi),
        \qquad
        |\mathscr C_\ell(\varpi,j)|
        =
        \sum_{\phi\in \mathscr A_j:\,\phi(t_\ell^j)=s_j} w(\phi).
\]
We claim that the weights $w(\phi)$ are asymptotically independent of
$\phi$.  Let $B_{\rm emb}$ be the set of indices of the inactive fans already
embedded by $\varpi$, let
\[
        k'=\bigl|B\setminus(B_{\rm emb}\cup\{j\})\bigr|,
        \qquad
        N=k'(v(\mathcal F_r)-1).
\]
If $k'=0$, then $w(\phi)=1$ for every $\phi\in\mathscr A_j$, and the claim
follows from Lemma~\ref{lem:single-fan}.  Hence assume $k'>0$.

For each still unembedded fan $\mathcal F_r^c$, record the vertex mapped to
$s_c$ and, for every other vertex, the part containing its image.  Let
$\Sigma$ be the finite set of records for which the individual fan edges lie
in $\mathcal J$ and the transversal-coordinate condition from
Claim~\ref{CLAIM-2} is satisfied.  The set $\Sigma$ depends only on the fixed
data, not on the actual partite labels used by $\phi$.  Since
$\mathscr C(\varpi)\ne\emptyset$, at least one such record is feasible, so
$\Sigma\ne\emptyset$.

For $\xi\in\Sigma$, let $\lambda_m(\xi)$ be the number of still unembedded
partite vertices prescribed by $\xi$ to lie in $V_m$.  Then
$\sum_m\lambda_m(\xi)=N$.  For a fixed $\xi$ and $\phi$, the number of
extensions realizing $\xi$ is
\[
        \prod_{m=1}^r
        \bigl(|V_m|-a_m(\phi,\xi)\bigr)_{\lambda_m(\xi)},
\]
where $a_m(\phi,\xi)=O_{r,t}(1)$ accounts for the already unavailable vertices
in $V_m$.  Since $|V_m|=n/r+O_{r,t}(1)$, this product equals
\[
        (1+O_{r,t}(n^{-1}))
        \prod_{m=1}^r |V_m|^{\lambda_m(\xi)}
\]
uniformly in $\phi$.  Summing over $\xi\in\Sigma$ gives
\[
        w(\phi)=(1+O_{r,t}(n^{-1}))\gamma n^N,
\]
uniformly for $\phi\in\mathscr A_j$, where $\gamma>0$ is independent of
$\phi$.

Consequently, by Lemma~\ref{lem:single-fan},
\[
        \frac{|\mathscr C_\ell(\varpi,j)|}
             {|\mathscr C(\varpi)|}
        =
        \frac{\sum_{\phi\in \mathscr A_j:\,\phi(t_\ell^j)=s_j}w(\phi)}
             {\sum_{\phi\in \mathscr A_j}w(\phi)}
        \ge
        \frac{\vartheta_r}{2}
        =
        \theta_r
\]
for all sufficiently large $n$.
\end{proof}

\medskip
We now finish the count. For $\psi\in\mathscr C_0$ and $\ell\ne k$, say that an
inactive fan $\mathcal F_r^j$, $j\in B$, is \emph{$\ell$-good} if $\psi(t_\ell^j)=s_j$.
If $\psi$ has an $\ell$-good inactive fan for every $\ell\ne k$, choose
one such fan and denote its index by $j_\ell$.  These indices are distinct:
one inactive fan uses exactly one vertex of $S$, so it cannot be good in two
different coordinates.  The blowup transversal edge $\{t_k^i\}\cup\{t_\ell^{j_\ell}:\ell\ne k\}$
is mapped to $\{z\}\cup\{s_{j_\ell}:\ell\ne k\}$, which belongs to
$\mathcal S(Z)$ because $z\in Z$ and
$\{s_{j_\ell}:\ell\ne k\}\in\binom{S}{r-1}$.  Hence every surviving member of
$\mathscr C_0$ has some coordinate $\ell\ne k$ for which no inactive fan is
$\ell$-good.

Fix $\ell\ne k$. Applying Claim~\ref{CLAIM-3} successively to the inactive fans gives
\[
        |\{\psi\in\mathscr C_0:
          \text{no inactive fan is }\ell\text{-good}\}|
        \le
        (1-\theta_r)^{t-1}|\mathscr C_0|.
\]
Taking the union over the $r-1$ choices of $\ell\ne k$, the proportion of
surviving members of $\mathscr C_0$ is at most $(r-1)(1-\theta_r)^{t-1}$. This estimate is uniform over the initial choice of the active fan, its
embedding, and the assignment of the vertices of $S$ to the inactive fans.
Together with Claim~\ref{CLAIM-1}, which rules out surviving copies in which $e$ is the
active transversal edge, we obtain
\[
        N_{\mathcal F}\bigl((\mathcal J-\mathcal S(Z))\cup\{e\};e\bigr)
        \le
        (r-1)(1-\theta_r)^{t-1}
        N_{\mathcal F}(\mathcal J\cup\{e\};e).
\]
The lemma follows with $c_r=r$.
\end{proof}

\subsection{Proof of the supersaturation theorem}

We now assemble the construction. The selection lemma below supplies enough
members of $\overline{\mathcal J}_{\mathrm{par}}$ of any prescribed
non-crossing type inside a small set
$Z$, so that the deleted $S$-star edges can be replaced and the total number
of edges becomes exactly $\ex(n,\mathcal F)+q$.

We use the following elementary selection lemma, with $m_S=\binom{t-1}{r-1}$ as above.

\begin{lemma}\label{lem:selection}
Let $a_1,\ldots,a_r$ be nonnegative integers with
$\sum_i a_i=r$ and $(a_1,\ldots,a_r)\ne(1,\ldots,1)$. There is a constant
$b_0=b_0(r,t)>0$ such that the following holds for all sufficiently large $n$ and for every integer $q$ with $1\le q\le n$. There are subsets
$Z_i\subseteq V_i$ of a common size $\ell=O_{r,t}(q^{1/r})$ such that at least
$q+m_Sr\ell$ members $e\in \overline{\mathcal J}_{\mathrm{par}}$ with
$e\subseteq Z_1\cup\cdots\cup Z_r$ satisfy
$|e\cap V_i|=a_i$ for every $i\in[r]$.
Moreover,
\(
        q+m_Sr\ell\le b_0q.
\)
The constant $b_0$ may be chosen to satisfy $b_0=O_r(t^r)$.
\end{lemma}

\begin{proof}
For fixed subsets $Z_i\subseteq V_i$ of common size $\ell$, the number of $r$-sets
$e\subseteq Z_1\cup\cdots\cup Z_r$ with
$|e\cap V_i|=a_i$ for every $i$ is
\[
        p_{\ell}:=\prod_{i=1}^r \binom{\ell}{a_i},
\]
where a factor with $a_i=0$ is interpreted as $1$. Since there are only finitely
many vectors $(a_1,\ldots,a_r)$ with $\sum_i a_i=r$, there are constants
$c_1>0$ and $\ell_0=\ell_0(r)$ such that
$p_{\ell}\ge c_1\ell^r$ for every $\ell\ge \ell_0$ and for every such vector.

Choose a constant $c_0=c_0(r,t)\ge \ell_0$ large enough that
\[
        c_1c_0^r\ge 1+m_Sr(c_0+1).
\]
Since $m_S=\binom{t-1}{r-1}=O_r(t^{r-1})$, we may take $c_0=O_r(t)$. Put
\(
        \ell:=\left\lceil c_0q^{1/r}\right\rceil.
\)
Then 
\[\ell\le c_0q^{1/r}+1\le (c_0+1)q\]
because $q\ge1$. Also,
\[
        p_{\ell}\ge c_1\ell^r\ge c_1c_0^rq
        \ge q+m_Sr(c_0+1)q
        \ge q+m_Sr\ell.
\]
Thus there are at least $q+m_Sr\ell$ members of
$\overline{\mathcal J}_{\mathrm{par}}$ of the required type inside
$Z_1\cup\cdots\cup Z_r$. Moreover,
\[
        q+m_Sr\ell\le (1+m_Sr(c_0+1))q.
\]
Hence the conclusion holds with
\[
        b_0:=1+m_Sr(c_0+1)=O_r(t^r).
\]
Finally, since $q\le n$, we have
$\ell=O_{r,t}(n^{1/r})=o(n)$, while each balanced part has size
$n/r+O_{r,t}(1)$.  Thus the sets $Z_i$ fit inside the parts for all
sufficiently large $n$.
\end{proof}

\begin{proof}[Proof of Theorem~\ref{thm:main}]
Fix $r\ge 2$ and $K>1$. For each $t\ge r$, let $b_0=b_0(r,t)$ be the constant
from Lemma~\ref{lem:selection}. By that lemma, $b_0=O_r(t^r)$. On the other hand,
Lemma~\ref{lem:star-deletion} gives $\rho_t=c_r(1-\theta_r)^{t-1}$, which decays exponentially in $t$. Hence we may choose $t\ge r$ so large that
\[
        b_0\rho_t\le \frac1{2K}.
\]
Set $\mathcal{F}:=\mathcal{F}_{r,t}$. By Lemma~\ref{thm:weak-stability} and
Theorem~\ref{thm:exact}, both proved in Section~\ref{sec:exact},
$\mathcal{F}$ is stable and, for all sufficiently large $n$, its unique
extremal $r$-graph is
$\mathcal{J}=\mathcal{J}_{t-1}(n)$.
Let $a_{\mathcal F}$ be the constant from Lemma~\ref{lem:two-new}. By
Lemma~\ref{lem:c-scale}, there is a constant $b_{\mathcal F}>0$ such that
\(
        c(n,\mathcal F)\ge b_{\mathcal F}n^{m_{\mathcal F}}
\)
for all sufficiently large $n$. Choose
\[
        \delta_{\mathcal F}\le
        \min\left\{1,\frac{b_{\mathcal F}}{2Ka_{\mathcal F}b_0^2}\right\}.
\]
Now let $n$ be sufficiently large and let $q$ be an integer with $1\le q\le \delta_{\mathcal F}n$. In particular, $q\le n$.

Choose $e^*\in\overline{\mathcal J}_{\mathrm{par}}$ such that
\[
        N_{\mathcal F}(\mathcal J\cup\{e^*\};e^*)=c(n,\mathcal F),
\]
and put $a_i:=|e^*\cap V_i|$ for $i\in[r]$.  Since $e^*$ is a partite
non-edge, $(a_1,\ldots,a_r)\ne(1,\ldots,1)$.  Lemma~\ref{lem:selection}
gives sets $Z_i\subseteq V_i$ of common size
$\ell=O_{r,t}(q^{1/r})$.  Let $Z:=Z_1\cup\cdots\cup Z_r$.  There are at least
$q+m_S|Z|=q+m_Sr\ell$ members
$e\in\overline{\mathcal J}_{\mathrm{par}}$ with $e\subseteq Z$ and
$|e\cap V_i|=a_i$ for every $i\in[r]$, and
$q+m_S|Z|\le b_0q$.  Choose distinct such members
$e_1,\ldots,e_{m_1}$, where
\[
        m_1:=q+m_S|Z|,
\]
and set
\[
        \mathcal H:=(\mathcal J-\mathcal S(Z))\cup\{e_1,\ldots,e_{m_1}\}.
\]
Since $|\mathcal S(Z)|=m_S|Z|$ and $|\mathcal J|=\ex(n,\mathcal F)$ by
Theorem~\ref{thm:exact}, we have
\[
        |\mathcal H|=|\mathcal J|-m_S|Z|+m_1=\ex(n,\mathcal F)+q.
\]
Moreover, $m_1\le b_0q$.

It remains to count copies of $\mathcal{F}$ in $\mathcal{H}$. Since
$\mathcal J$ is $\mathcal F$-free, every copy of $\mathcal F$ in
$\mathcal H$ uses at least one new edge among $e_1,\ldots,e_{m_1}$. First consider copies using exactly one new edge $e_i$.  Such a copy is
contained in $(\mathcal J-\mathcal S(Z))\cup\{e_i\}$ and contains $e_i$.
Since $|e_i\cap V_j|=|e^*\cap V_j|$ for every $j\in[r]$, an automorphism of
$\mathcal J$ preserving $S$ and each part setwise maps $e^*$ to $e_i$.  Hence
\[
        N_{\mathcal F}(\mathcal J\cup\{e_i\};e_i)=c(n,\mathcal F).
\]
As $e_i\subseteq Z$, Lemma~\ref{lem:star-deletion} gives
\[
        N_{\mathcal F}((\mathcal J-\mathcal S(Z))\cup\{e_i\};e_i)
        \le \rho_t c(n,\mathcal F).
\]
Therefore the number of copies using exactly one new edge is at most
\[
        m_1\rho_t c(n,\mathcal F)
        \le b_0q\rho_t c(n,\mathcal F)
        \le \frac1{2K}q c(n,\mathcal F).
\]
Next, by Lemma~\ref{lem:two-new}, the number of copies using at least two new
edges is at most
\[
        a_{\mathcal{F}}m_1^2n^{m_{\mathcal F}-1}\le
        a_{\mathcal{F}}b_0^2q^2n^{m_{\mathcal F}-1}.
\]
Using $c(n,\mathcal F)\ge b_{\mathcal F}n^{m_{\mathcal F}}$ and
$q\le\delta_{\mathcal F}n$, we get
\[
        a_{\mathcal{F}}b_0^2q^2n^{m_{\mathcal F}-1}
        \le
        \frac{a_{\mathcal F}b_0^2}{b_{\mathcal F}}\frac qn\,q c(n,\mathcal F)
        \le
        \frac1{2K}q c(n,\mathcal F).
\]
Combining the two estimates yields $N_{\mathcal{F}}(\mathcal{H})\le K^{-1}q c(n,\mathcal{F})$. Since $|\mathcal{H}|=\ex(n,\mathcal{F})+q$, we have
\[
        h_{\mathcal{F}}(n,q)\le N_{\mathcal{F}}(\mathcal{H})\le
        K^{-1}q c(n,\mathcal{F}).
\]
The theorem follows.
\end{proof}

\section{Exact Tur\'an theorem for \texorpdfstring{$\mathcal{F}_{r,t}$}{Frt}}\label{sec:exact}
This section proves Theorem~\ref{thm:exact}. Before entering the hypergraph
part of the proof, we isolate the standard stability tools and settle the graph
case. The stability step is an Erd\H{o}s--Simonovits-type stability statement for
$\mathcal F_{r,t}$, obtained from the blowup and removal lemmas together with
the known stability theorem for the fan $\mathcal F_r$. The graph case $r=2$ is
then handled separately by Simonovits' exact theorem and the Erd\H{o}s--Gallai
theorem for matchings. After these reductions, the rest of the section assumes
$r\ge3$ and proves that every deviation from the join construction either
creates a copy of $\mathcal F_{r,t}$ or is too sparse to be extremal.

We first record the two standard lemmas used to derive
Erd\H{o}s--Simonovits-type stability.

\begin{lemma}[\cite{Erdos1964, ErdosSimonovits1983}]\label{lem:blowup}
Let $\mathcal{F}$ be a fixed $r$-graph, and fix a positive integer $s$. If an $n$-vertex $r$-graph
$\mathcal{H}$ contains $\Omega(n^{v(\mathcal{F})})$ copies of $\mathcal{F}$, then, for all sufficiently
large $n$, $\mathcal{H}$ contains a copy of the blowup $\mathcal{F}[s]$.
\end{lemma}

\begin{lemma}[\cite{Gowers2007,KomlosSimonovits1996,NagleRodlSchacht2006,RodlSkokan2006,Tao2006}]\label{lem:removal}
Let $r\ge 2$ and let $\mathcal{F}$ be a fixed $r$-graph. For every $\eta>0$ there are constants $\xi>0$ and $n_0$ such that the following holds. If $n\ge n_0$ and an $n$-vertex $r$-graph $\mathcal{H}$ contains at most $\xi n^{v(\mathcal{F})}$ copies of $\mathcal{F}$, then deleting at most
$\eta n^r$ edges from $\mathcal{H}$ makes it $\mathcal{F}$-free.
\end{lemma}

We now derive the Erd\H{o}s--Simonovits-type stability consequence needed
below. The Mantel theorem
and its generalization by Mubayi and Pikhurko say that the fan $\mathcal F_r$
has the balanced complete $r$-partite $r$-graph as its Tur\'an extremal
construction, and the required stability theorem is classical for
$r=2$~\cite{ErdosSimonovits1966} and due to Mubayi and Pikhurko for
$r\ge3$~\cite{MubayiPikhurko2007}.

\begin{lemma}\label{thm:weak-stability}
Fix $r\ge 2$ and $t\ge r$. For every $\delta>0$ there are constants $\eps>0$ and $n_0$ such that every $\mathcal{F}_{r,t}$-free $r$-graph $\mathcal{H}$ on $n\ge n_0$ vertices with $|\mathcal{H}|\ge t_r^{(r)}(n)-\eps n^r$
is $\delta n^r$-close to a copy of $\mathcal{T}_r^{(r)}(n)$.
\end{lemma}

\begin{proof}
By the stability theorem of Erd\H{o}s--Simonovits for $r=2$ and Mubayi--Pikhurko
for $r\ge3$, choose $\eta>0$ such that every $\mathcal F_r$-free $r$-graph
$\mathcal H'$ on $n$ sufficiently large vertices with
$|\mathcal H'|\ge t_r^{(r)}(n)-\eta n^r$ is $\frac{\delta}{2}n^r$-close to a
copy of $\mathcal T_r^{(r)}(n)$. Put $\tau:=\min\{\eta/2,\delta/2\}$ and $\eps:=\eta/2$.
Let $\mathcal{H}$ be an $\mathcal{F}_{r,t}$-free $r$-graph with $|\mathcal{H}|\ge t_r^{(r)}(n)-\eps n^r$. Since
$\mathcal{F}_{r,t}\subseteq \mathcal{F}_r[t]$, Lemma~\ref{lem:blowup} implies that $\mathcal{H}$ contains
$o(n^{v(\mathcal{F}_r)})$ copies of $\mathcal{F}_r$. Otherwise, $\mathcal{H}$ would contain $\mathcal{F}_r[t]$, and then 
$\mathcal{F}_{r,t}\subseteq \mathcal{H}$. By Lemma~\ref{lem:removal}, deleting at most $\tau n^r$ edges gives an $\mathcal{F}_r$-free $r$-graph $\mathcal{H}_0$. Moreover,
\[
        |\mathcal{H}_0|\ge t_r^{(r)}(n)-(
        \eps+\tau)n^r\ge t_r^{(r)}(n)-\eta n^r.
\]
The stability theorem gives a copy of $\mathcal{T}_r^{(r)}(n)$ on
$V(\mathcal{H})$ such that $|\mathcal{H}_0\triangle \mathcal{T}_r^{(r)}(n)|\le \frac{\delta}{2}n^r$. Adding back the deleted edges changes the symmetric difference by at most $\tau n^r$, and so
\[
        |\mathcal{H}\triangle \mathcal{T}_r^{(r)}(n)|
        \le \tau n^r+\frac{\delta}{2}n^r\le \delta n^r.
\]
Thus, $\mathcal{H}$ is $\delta n^r$-close to the complete balanced $r$-partite $r$-graph $\mathcal{T}_r^{(r)}(n)$.
\end{proof}

We next dispose of the graph case, so that the numbered subsections below may
focus on $r\ge3$. When $r=2$, the graph
$\mathcal F_{2,t}$ is $3$-chromatic. Its decomposition family in
Simonovits' theorem~\cite{Simonovits1968} has the unique inclusion-minimal member $M_t$, the matching of size $t$. Indeed, deleting any one color class from a proper $3$-coloring leaves at least one edge from each of the $t$ vertex-disjoint triangles, and hence leaves a graph containing a matching of size $t$. Conversely, there is a proper coloring in which the deleted color class is
$\{t_1^i:i\in[t]\}$; after deleting this class, the remaining graph is exactly
the matching $\{c^i t_2^i:i\in[t]\}$ together with isolated vertices. Hence the
unique minimal member of the decomposition family of $\mathcal{F}_{2,t}$ is the matching
$M_t$.

Simonovits' exact theorem for $3$-chromatic graphs whose decomposition family contains a matching~\cite{Simonovits1968} reduces the large-$n$ extremal problem to adding internal graphs to the two classes of a complete bipartite graph, with total matching number at most $t-1$.
% Simonovits' exact theorem for $3$-chromatic graphs with finite decomposition
% family~\cite{Simonovits1968} reduces the large-$n$ extremal problem to adding
% internal graphs to the two classes of a complete bipartite graph, with total
% matching number at most $t-1$. 
Indeed, if the two internal graphs contained
disjoint matchings of total size $t$, then these matching edges, together with
the complete bipartite graph between the two classes, would realize
$\mathcal F_{2,t}$. The Erd\H{o}s--Gallai matching theorem
\cite{ErdosGallai1959}, together with its equality case for fixed matching
number and large side size, implies that the internal edges are maximized by
making them incident with a fixed set of $t-1$ vertices, possibly distributed
between the two sides. These $t-1$ vertices are full-degree vertices, and
optimizing the remaining bipartition leaves the two non-exceptional parts
balanced. Equivalently, the resulting graph is
\[
        K^{(2)}_{t-1}\vee \mathcal{T}_2^{(2)}(n-t+1)=\mathcal{J}_{t-1}(n).
\]
Thus Theorem~\ref{thm:exact} holds for $r=2$.

It remains to prove the exact theorem for $r\ge 3$. From now on, all constants
in this section may depend on the fixed integers $r$ and $t$, and $a\ll b$
denotes the usual hierarchy of small constants.

\subsection{Partition from the stability}

We begin the hypergraph exact proof by fixing a near-Tur\'an partition of an
extremal $\mathcal F_{r,t}$-free $r$-graph. The notation introduced here separates
missing crossing edges from bad non-crossing edges and records the first
quantitative consequences of stability.

We first check the proposed extremal construction itself and record the
edge-number gap gained by adding one more exceptional vertex.

\begin{lemma}\label{lem:deterministic}
The $r$-graph $\mathcal J_{t-1}(n)$ is $\mathcal F_{r,t}$-free. Moreover,
there is a constant $c_{r,t}>0$ such that, for all sufficiently large $n$ and all
$0\le s\le t-2$,
\[
        |\mathcal J_{s+1}(n)|-|\mathcal J_s(n)|\ge c_{r,t}n^{r-1}.
\]
Consequently,
\[
        |\mathcal J_{t-1}(n)|-|\mathcal J_s(n)|
        \ge (t-1-s)c_{r,t}n^{r-1}
        \qquad\text{for}\quad 0\le s\le t-1.
\]
\end{lemma}

\begin{proof}
Let $S$ be the set of the $t-1$ full-degree vertices of
$\mathcal J_{t-1}(n)$. The subgraph $\mathcal J_{t-1}(n)-S$ is $r$-partite,
and contains no copy of $\mathcal F_r$: the transversal edge of
$\mathcal F_r$ forces the transversal vertices into distinct parts, after
which the center cannot be placed so that all petals are crossing. Hence each
of the $t$ vertex-disjoint copies of $\mathcal F_r$ inside
$\mathcal F_{r,t}$ must meet $S$, which is impossible because $|S|=t-1$.

It remains to prove the gap estimate. Put $m=n-s$. Since
\[
        |\mathcal J_s(n)|
        =
        \binom nr-\binom mr+t_r^{(r)}(m),
\]
we have
\[
        |\mathcal J_{s+1}(n)|-|\mathcal J_s(n)|
        =
        \binom{m-1}{r-1}
        -
        \bigl(t_r^{(r)}(m)-t_r^{(r)}(m-1)\bigr).
\]
Also
\[
        t_r^{(r)}(m)-t_r^{(r)}(m-1)
        =
        r^{-(r-1)}m^{r-1}+O_r(m^{r-2}).
\]
Consequently,
\[
        |\mathcal J_{s+1}(n)|-|\mathcal J_s(n)|
        =
        \left(\frac1{(r-1)!}-\frac1{r^{r-1}}+o(1)\right)m^{r-1}.
\]
The coefficient is positive for $r\ge3$, and $m=n-O_{r,t}(1)$, so the first
inequality follows by choosing $c_{r,t}>0$ sufficiently small. Summing the first inequality for $s'=s,s+1,\ldots,t-2$ gives the displayed consequence.
\end{proof}

Let $0<\alpha<1$ be fixed later and let $\mathcal{H}$ be an extremal
$\mathcal{F}_{r,t}$-free $r$-graph on $n$ vertices. Since
\[
        |\mathcal{H}|\ge |\mathcal{J}_{t-1}(n)|
        =t_r^{(r)}(n)+\Theta_{r,t}(n^{r-1}),
\]
Lemma~\ref{thm:weak-stability} gives
that $\mathcal{H}$ is $(\alpha/r)n^r$-close to a copy of $\mathcal{T}_r^{(r)}(n)$. Let $\sigma=(V_1,\dots,V_r)$ be an $r$-partition maximizing
\[
        f_{\mathcal{H}}(\sigma):=\sum_{e\in \mathcal{H}}
        |\{i\in [r]:e\cap V_i\ne\varnothing\}|.
\]
Let $\mathcal{T}_\sigma$ be the complete $r$-partite $r$-graph on $V(\mathcal{H})$ with parts
$V_1,\dots,V_r$. We call the edges in $\mathcal{T}_\sigma\setminus \mathcal{H}$ \emph{missing},
and the edges in $\mathcal{H}\setminus \mathcal{T}_\sigma$ \emph{bad}. If $\tau$ is the
$r$-partition given by the above copy of $\mathcal{T}_r^{(r)}(n)$, then by the maximality
of $\sigma$,
\[
        r\left(|\mathcal{H}|-\frac{\alpha}{r}n^r\right)
        \le f_{\mathcal{H}}(\tau)\le f_{\mathcal{H}}(\sigma)
        \le r|\mathcal{H}|-|\mathcal{H}\setminus \mathcal{T}_\sigma|,
\]
which yields $|\mathcal{H}\setminus \mathcal{T}_\sigma|\le \alpha n^r$.
On the other hand, since $\mathcal{T}_\sigma$ is $\mathcal{F}_{r,t}$-free, the extremality of $\mathcal{H}$
gives $|\mathcal{H}|\ge |\mathcal{T}_\sigma|$. Consequently,
\[
        |\mathcal{T}_\sigma\setminus \mathcal{H}|
        =|\mathcal{T}_\sigma|-|\mathcal{H}|+|\mathcal{H}\setminus \mathcal{T}_\sigma|
        \le |\mathcal{H}\setminus \mathcal{T}_\sigma|
        \le \alpha n^r.
\]
Moreover, since 
$|\mathcal{T}_\sigma|\ge |\mathcal{H}|-|\mathcal{H}\setminus \mathcal{T}_\sigma|$, a standard estimate shows
that, for each $i\in [r]$,
\begin{equation}\label{eq:part-balance}
    \frac{n}{r}-\alpha^{1/r}n
        \le |V_i|\le
        \frac{n}{r}+\alpha^{1/r}n.
\end{equation}

We write $d_{\mathcal H}(U)$ for the number of edges of $\mathcal H$
containing a set $U$, and write $\delta_1(\mathcal H)$ for the minimum
vertex degree of $\mathcal H$. For a vertex $v$, let $d_B(v)$ and $d_M(v)$ denote, respectively, the number
of bad edges and missing crossing edges containing $v$. We use the analogous
notation $d_B(u,v)$, $d_M(u,v)$ and $d_B(u,v,w)$ for pairs and triples.

Let $h:=v(\mathcal F_r)$. Set
\[
        k_{\rm aux}:=1000trh,
        \qquad\text{and}\qquad d:=10k_{\rm aux}\binom{n}{r-3}.
\]
The constant $k_{\rm aux}$ bounds all auxiliary forbidden sets used in the fan-completion arguments. A pair $\{u,v\}$ whose vertices lie in distinct parts is called a \emph{cross-pair}, and it is \emph{sparse} if $d_{\mathcal{H}}(u,v)\le d$. Define
\[
        L:=\{v: v\text{ is contained in at least }
        \alpha^{1/3}n\text{ sparse pairs}\}.
\]

Sparse cross-pairs are pairs that cannot be used freely in the fan-completion
argument. The following bound shows that they are rare, and therefore that the
vertices incident with many of them form a small exceptional set.

\begin{lemma}\label{lem:sparse}
The number of sparse pairs is at most $b_{r,t}\alpha n^2$, where $b_{r,t}>0$ depends only on $r$ and $t$. In particular,
$|L|\le b_{r,t}\alpha^{2/3}n$.
\end{lemma}

\begin{proof}
By the part-size estimate \eqref{eq:part-balance}, every cross-pair is contained in at least
$c_r n^{r-2}$ crossing $r$-sets of $\mathcal{T}_\sigma$, for some $c_r>0$ and all sufficiently large $n$. If a cross-pair is sparse, then only $O_{r,t}(n^{r-3})$ of these crossing sets are present in $\mathcal{H}$. Hence, for all large $n$, the pair
is contained in at least $(c_r/2)n^{r-2}$ missing crossing edges. Since a missing
edge contains at most $\binom{r}{2}$ cross-pairs, and $|\mathcal{T}_\sigma\setminus \mathcal{H}|\le \alpha n^r$, the number of sparse pairs is at most $b_{r,t}\alpha n^2$.
Since every vertex of $L$ is incident with at least $\alpha^{1/3}n$ sparse pairs, 
counting incidences between vertices of $L$ and the sparse pairs gives
$|L|\alpha^{1/3}n\le 2b_{r,t}\alpha n^2$, and hence the stated bound.
\end{proof}

\subsection{Completion of one fan}

The next tool is a local fan-completion lemma. It says that once a transversal
edge and one suitable petal are available, sufficiently many pair-degrees allow
us to greedily complete the remaining petals while avoiding a bounded forbidden
set.

\begin{lemma}\label{lem:fan-completion}
Fix $R\subseteq V(\mathcal{H})$ with $|R|\le k_{\rm aux}$, and fix
$z\in V_i\setminus R$. Let $T=\{x_1, \dots,x_r\}\in \mathcal{H}-R$, and $x_j\in V_j$ with $x_i\ne z$. Suppose that one of the following holds:
\begin{enumerate}[label=(\alph*)]
\item there is a bad edge $P_i\in \mathcal{H}-R$ containing $\{z,x_i\}$ and
satisfying $P_i\cap T=\{x_i\}$;
\item $d_B(z,x_i)\ge \gamma_0 n^{r-2}$ for some fixed $\gamma_0>0$.
\end{enumerate}
Assume also that $d_{\mathcal{H}}(z,x_j)>d$ for every $j\ne i$. Then $\mathcal{H}-R$ contains a copy of $\mathcal{F}_r$ with center $z$, and transversal edge $T$, and in case $(a)$, its $i$th petal is the prescribed edge $P_i$.
\end{lemma}

\begin{proof}
In case (a), keep the edge $P_i$ fixed. In case (b), only $O_{r,t}(n^{r-3})$ bad edges
through $\{z,x_i\}$ meet $(R\cup T)\setminus \{z,x_i\}$. Since $\gamma_0$ is fixed, for all sufficiently large $n$ there is a bad edge $P_i$ through $\{z,x_i\}$ such that $P_i\cap(R\cup T)=\{x_i\}$. Thus in either case we have a
fixed petal $P_i\in \mathcal{H}-R$ containing $\{z,x_i\}$ and meeting $T$ exactly
in $x_i$. Now we choose the remaining petals greedily. Suppose petals have already been chosen for the coordinates in $I\subseteq [r]\setminus\{i\}$, and let $R'$ be the union of
$R$, $T$, $P_i$, and the petals already chosen. Then $|R'|\le k_{\rm aux}+rh+r<2k_{\rm aux}$. For the next coordinate $j\notin I\cup\{i\}$, the pair $\{z,x_j\}$ lies in more than
$d=10k_{\rm aux}\binom{n}{r-3}$ edges of $\mathcal{H}$, while fewer than $|R'|\binom{n}{r-3}<d$ of these edges meet $R'\setminus\{z,x_j\}$. Hence, we can
choose a new petal through $\{z,x_j\}$ that is disjoint from all previous choices outside $z$.
Repeating this for every remaining coordinate completes the desired fan with transversal edge $T$.
\end{proof}

The next proposition is the main local consequence of a large bad degree. It
associates to such a vertex a small seed set and large candidate sets, with the
property that every transversal edge through the candidate sets can be completed
to a fan avoiding any bounded forbidden set.

\begin{prop}\label{prop:heavy-package}
For every fixed $\mu>0$ there exist constants
$\gamma_h(\mu)>0$ and $\alpha_h(\mu)>0$ such that the following holds. Suppose $0<\alpha\le \alpha_h(\mu)$, let $n$ be sufficiently large, and let $v\in V_i$ satisfy $d_B(v)\ge \mu n^{r-1}$.
Then, for every $R\subseteq V(\mathcal H)\setminus\{v\}$ with
$|R|\le k_{\rm aux}$, one can find a set $C=C(v;R)$ and sets $X_1=X_1(v;R),\ldots,X_r=X_r(v;R)$ such that $v\in C\subseteq V(\mathcal H)\setminus R$, $|C|\le 2$ and $X_j\subseteq V_j\setminus (R\cup C)$, $|X_j|\ge \gamma_h(\mu) n$ for every $j \in [r]$. 
Moreover, these sets are robust in the following sense. For every
$U\subseteq V(\mathcal H)\setminus C$ with $|R\cup U|\le k_{\rm aux}$,
every edge $T=\{x_1,\ldots,x_r\}\in \mathcal H-(R\cup U)$ satisfying
$x_j\in X_j\setminus U$ for every $j\in[r]$
extends in $\mathcal H-(R\cup U)$ to a copy of $\mathcal F_r$ that
contains $C$ and has $T$ as its transversal edge.
\end{prop}

\begin{proof}
Temporarily relabel the parts so that $v\in V_1$. For $j\in [r]$, put
\[
        b_j(v):=|\{e\in \mathcal{H}\setminus \mathcal{T}_\sigma:
        v\in e \text{ and }
        |e\cap V_j|\ge 2\}|.
\]
Choose $\kappa=\kappa(r,t,\mu)>0$ so small that
$((r-1)+10)\kappa<\mu/2$. Since every bad edge through $v$ is counted by at
least one of the quantities $b_j(v)$, the following dichotomy holds: either
\begin{equation}\label{eq:heavy-cross-alternative}
        b_j(v)\ge \kappa n^{r-1}
        \quad\text{for some }j\in\{2,\dots,r\},
\end{equation}
or
\begin{equation}\label{eq:heavy-same-alternative}
        b_1(v)\ge 10\kappa n^{r-1}.
\end{equation}
Indeed, if both alternatives failed, then
$d_B(v)\le \sum_j b_j(v)<((r-1)+10)\kappa n^{r-1}<\mu n^{r-1}$, contrary to the hypothesis. Choose $\gamma_h(\mu)>0$ with $\gamma_h(\mu)\le \min\{\kappa/16,1/(6r)\}$. We shall also shrink $\alpha_h(\mu)>0$ whenever needed below.

First assume \eqref{eq:heavy-cross-alternative} holds. After relabelling
$V_2,\dots,V_r$, suppose that $b_2(v)\ge \kappa n^{r-1}$. For pairs
$\{u,w\}\subseteq V_2$, write
\[
        \omega(u):=\sum_{w\in V_2\setminus\{u\}}d_B(v,u,w).
\]
Every bad edge counted by $b_2(v)$ contains at least one pair from $V_2$, and
hence
\[
        \sum_{\{u,w\}\subseteq V_2}d_B(v,u,w)\ge b_2(v)\ge \kappa n^{r-1}.
\]
Thus $\sum_{u\in V_2}\omega(u)\ge 2\kappa n^{r-1}$. Vertices in
$(L\cup R)\cap V_2$ contribute at most $(|L|+k_{\rm aux})n^{r-2}$. By Lemma~\ref{lem:sparse}, after shrinking $\alpha_h(\mu)$ and then taking
$n$ large, this contribution is at most $\kappa n^{r-1}$. Therefore some $u\in V_2\setminus(L\cup R)$ satisfies
$\omega(u)\ge \kappa n^{r-2}/2$.
Fix such a vertex $u$, and define
\[
        Y_2:=\{x\in V_2\setminus(R\cup\{u,v\}):
        d_B(v,u,x)\ge \kappa n^{r-3}/8\}.
\]
If $|Y_2|<\kappa n/8$, then
\[
        \omega(u)
        \le |Y_2|n^{r-3}+n\cdot \frac{\kappa}{8}n^{r-3}+O_{r,t}(n^{r-3})
        < \frac{\kappa}{2}n^{r-2}
\]
for all sufficiently large $n$, a contradiction. Hence $|Y_2|\ge \kappa n/8$.
For $j\ne 2$, set
\[
        Y_j:=\{x\in V_j\setminus(R\cup\{u,v\}):d_{\mathcal{H}}(u,x)>d\}.
\]
Since $u\notin L$, it is incident with fewer than $\alpha^{1/3}n$ sparse pairs
in total. The part-size estimate \eqref{eq:part-balance} then gives
\[
        |Y_j|\ge |V_j|-\alpha^{1/3}n-O_{r,t}(1)\ge n/(3r)
        \qquad\text{for } j\ne 2,
\]
after shrinking $\alpha_h(\mu)$ if necessary.
Put $C=\{v,u\}$, and choose
subsets $X_j\subseteq Y_j$ with $|X_j|\ge \gamma_h(\mu)n$ for every $j$.
Now let $U$ and $T=\{x_1,\dots,x_r\}$ be as in the statement. Because
$x_2\in X_2$, there are at least $\kappa n^{r-3}/8$ bad edges through
$\{v,u,x_2\}$. The number of these edges meeting
$(R\cup U\cup T)\setminus\{v,u,x_2\}$ is $O_{r,t}(n^{r-4})$ (and is zero when
$r=3$). Hence, for large $n$, there is a bad edge $P_2$ through
$\{v,u,x_2\}$ which avoids $R\cup U$ and satisfies $P_2\cap T=\{x_2\}$. For
every $j\ne 2$, the choice $x_j\in X_j\subseteq Y_j$ gives $d_{\mathcal{H}}(u,x_j)>d$.
Applying Lemma~\ref{lem:fan-completion} with forbidden set $R\cup U$, center
$u$, index $2$, fixed petal $P_2$, and transversal edge $T$, gives the desired
copy of $\mathcal{F}_r$.

It remains to treat the case in which \eqref{eq:heavy-cross-alternative}
fails. Then \eqref{eq:heavy-same-alternative} holds. Define
\[
        Y_1:=\{x\in V_1\setminus(R\cup\{v\}):
        d_B(v,x)\ge \kappa n^{r-2}\}.
\]
Every edge counted by $b_1(v)$ contains $v$ and at least one further vertex of
$V_1$, and therefore
\[
        \sum_{x\in V_1\setminus\{v\}}d_B(v,x)\ge b_1(v)
        \ge 10\kappa n^{r-1}.
\]
The set $Y_1$ has size at least $\kappa n/2$ for all large $n$. Otherwise the
last sum would be at most
\[
        (\kappa n/2)n^{r-2}+n\cdot \kappa n^{r-2}+O_{r,t}(n^{r-2})
        <10\kappa n^{r-1},
\]
a contradiction.
For $j\ge 2$, set
\[
        Y_j:=\{x\in V_j\setminus(R\cup\{v\}):d_{\mathcal{H}}(v,x)>d\}.
\]
We claim that $|Y_j|\ge \kappa n$ for every $j\ge 2$. Suppose not for some
$j$. Then
\[
        \sum_{x\in V_j} d_{\mathcal{H}}(v,x)
        \le |Y_j|n^{r-2}+|R|n^{r-2}+|V_j|d
        \le 2\kappa n^{r-1}
\]
for all sufficiently large $n$. Move $v$ from $V_1$ to $V_j$. An edge through
$v$ gains one represented part only if it contains another vertex of $V_1$ and
avoids $V_j$, thus it loses one represented part only if it contains a vertex of
$V_j$ and contains no other vertex of $V_1$. By maximality of $f_{\mathcal{H}}(\sigma)$,
the number of gaining edges is at most the number of losing edges. The losing
edges are bounded by the displayed sum. On the other hand, all bad edges
counted by $b_1(v)$ and avoiding $V_j$ are gaining edges, and there are at
least
\[
        b_1(v)-\sum_{x\in V_j}d_{\mathcal{H}}(v,x)>8\kappa n^{r-1}
\]
of them, a contradiction. 
Put $C=\{v\}$, and choose subsets $X_j\subseteq Y_j$ of size at least
$\gamma_h(\mu)n$. If $T=\{x_1,\dots,x_r\}\in \mathcal{H}-(R\cup U)$ with
$x_j\in X_j\setminus U$, then $d_B(v,x_1)\ge \kappa n^{r-2}$ and
$d_{\mathcal{H}}(v,x_j)>d$ for every $j\ge 2$. Lemma~\ref{lem:fan-completion}, applied
with forbidden set $R\cup U$, center $v$, index $1$, and case $(b)$ with
$\gamma_0=\kappa$, completes the fan.

Undoing the temporary relabelling proves the proposition.
\end{proof}

\subsection{Heavy vertices and the gap sequence}\label{Heavy}

This subsection isolates the vertices responsible for many bad edges. We show
that there can be at most $t-1$ such heavy vertices, introduce the gap
sequence measuring the gap from $\mathcal J_{t-1}(n)$, and derive degree
information needed to control the remaining bad edges.

We now fix the hierarchy of constants. Let $c_{r,t}>0$ be the constant from Lemma~\ref{lem:deterministic}, and put
\[
        c_{\rm bad}:=r^2t,
        \qquad
        \lambda_r:=
        \frac14\left(\frac1{(r-1)!}-\frac1{r^{r-1}}\right)>0,
        \qquad\text{and}\qquad
        \gamma_b:=\frac1{4r}.
\]
Choose
\begin{equation}\label{eq:delta-hierarchy}
        0<\delta
        \ll
        \min\{c_{r,t}/c_{\rm bad},c_{r,t},\lambda_r,\gamma_b^{rt},
        \gamma_h(\lambda_r)^{rt}\}.
\end{equation}
After $\delta$ has been fixed, choose $\alpha>0$ sufficiently small so that
\begin{equation}\label{eq:alpha-hierarchy}
        \begin{aligned}
        \alpha
        \ll
        \min\{\delta^r,c_{r,t}^r,\lambda_r^r,
        \gamma_b^{rt},\gamma_h(\delta)^{rt},
        \gamma_h(\lambda_r)^{rt}\},
        \quad\text{and}\quad
        \alpha
        \le \min\{\alpha_h(\delta),\alpha_h(\lambda_r)\}.
        \end{aligned}
\end{equation}
Let
\[
        \eta_*:=\frac12\min\{\gamma_b,\gamma_h(\lambda_r)\}.
\]
The choices of $\delta$ and then $\alpha$ are also made sufficiently small so
that, for each fixed constant $s=s(r,t)$ occurring below,
$s(\delta+\alpha^{1/r})$ satisfies the smallness requirement for the parameter
$\zeta$ in Lemma~\ref{lem:transversal-selection}, both with
$\eta=\gamma_b/2$ and with $\eta=\eta_*/2$.
Finally, $n$ is taken sufficiently large in terms of all previously fixed
constants.

Define
\[
        W:=\{v\in V(\mathcal{H}): d_B(v)\ge \delta n^{r-1}\}.
\]
Vertices in $W$ will be called \emph{heavy}, and vertices outside $W$ will be called \emph{non-heavy}. Whenever Proposition~\ref{prop:heavy-package} is applied to a heavy
vertex $v$ with forbidden set $R$, we call the resulting data
\[
        R,\qquad C(v;R),\qquad X_1(v;R),\ldots,X_r(v;R)
\]
a \emph{heavy package}. If this package is assigned index $a$, we write
\[
        R_a:=R,\qquad C_a:=C(v;R),\qquad
        X_{a,j}:=X_j(v;R)\quad\text{for }j\in[r].
\]
We call $C_a$ its \emph{seed set} and $X_{a,1},\ldots,X_{a,r}$ its \emph{candidate
sets}.

\begin{cor}\label{cor:heavy-bound}
The set of heavy vertices satisfies $|W|\le t-1$.
\end{cor}

\begin{proof}
Suppose, for a contradiction, that $w_1,\ldots,w_t\in W$ are distinct.
We construct heavy packages recursively. Assume that packages have
already been chosen for all $b<a$, and put
\[
        R_a:= \{w_{a+1},\ldots,w_t\} \cup \bigcup_{b<a}C_b.
\]
For each $b<a$, the vertex $w_a$ belonged to the forbidden set $R_b$, so Proposition~\ref{prop:heavy-package}
gave $C_b\subseteq V(\mathcal H)\setminus R_b$. Thus $w_a\notin C_b$ for all $b<a$, and hence $w_a\notin R_a$. Also $|R_a|<k_{\rm aux}$.
Applying Proposition~\ref{prop:heavy-package} with $\mu=\delta$ to
$w_a$ and $R_a$ gives the $a$th heavy package, namely
\[
        C_a=C(w_a;R_a),\qquad X_{a,j}=X_j(w_a;R_a)\quad\text{for }j\in[r].
\]
Since $C_a\cap R_a=\emptyset$, the new seed set avoids all previous seed sets and all
later vertices $w_{a+1},\dots,w_t$. Thus the recursion is valid, and the seed sets
$C_1,\dots,C_t$ are pairwise disjoint.
Let
\[
        C_{\rm seed}:=\bigcup_{a=1}^t C_a,
        \qquad
        Y_{a,j}:=X_{a,j}\setminus C_{\rm seed}.
\]
With $\eta:=\gamma_h(\delta)/2$, all sets $Y_{a,j}$ have size at least
$\eta n$ for all sufficiently large $n$. We now choose vertices $x_{a,j}\in Y_{a,j}$. There are at least $\eta^{rt}n^{rt}$ choices, while the choices identifying two selected vertices in the same part contribute only $O_{r,t}(n^{rt-1})$.

Fix $(a_1,\dots,a_r)\in[t]^r$. The number of choices for which 
$\{x_{a_1,1},\dots,x_{a_r,r}\}\notin \mathcal{H}$ is at most
\[
        |\mathcal{T}_\sigma\setminus\mathcal{H}|O_{r,t}(n^{rt-r})
        \le O_{r,t}(\alpha n^{rt}).
\]
Taking the union over the at most $t^r$ vectors and using $\alpha\ll_{r,t}\eta^{rt}$, we may choose the vertices $x_{a,j}$ so that they
are pairwise distinct within each part, and all required transversal edges are present.
We embed the fans greedily. At step $a$, let $S_a$ be the union of the vertices already used, the later seed sets $C_b$ with $b>a$, and the selected vertices $x_{b,j}$ with $b>a$ and $j\in[r]$. Then $S_a\cap C_a=\varnothing$, and $|R_a\cup S_a|\le k_{\rm aux}$. Moreover, $T_a:=\{x_{a,1},\dots,x_{a,r}\}$ is an edge of $\mathcal{H}-(R_a\cup S_a)$, with
$x_{a,j}\in X_{a,j}\setminus S_a$ for every $j\in[r]$. Hence the extension
property in Proposition~\ref{prop:heavy-package} gives a copy of
$\mathcal{F}_r$ in $\mathcal{H}-(R_a\cup S_a)$ containing $C_a$ and with
transversal edge $T_a$.
After all $t$ steps, we have $t$ vertex-disjoint fans, and the selected transversal vertices 
supply all blowup transversal edges between them. Thus $\mathcal{H}$ contains
a copy of $\mathcal{F}_{r,t}$, a contradiction.\qedhere
\end{proof}

Define the \emph{gap sequence}
\[
        \Delta_n:=\ex(n,\mathcal{F}_{r,t})-|\mathcal{J}_{t-1}(n)|.
\]
By Lemma~\ref{lem:deterministic}, $\Delta_n\ge 0$. We prove that there is
$n_0$ such that, for every $n\ge n_0$,
\begin{equation}\label{eq:gap-step}
        \Delta_n\ge \Delta_{n-1}\quad\Longrightarrow\quad \Delta_n=0.
\end{equation}
This implication is enough for the eventual exact result. Indeed, if
$\Delta_n>0$ for some $n\ge n_0$, then \eqref{eq:gap-step} gives
$\Delta_n<\Delta_{n-1}$. Hence, as long as the sequence remains positive in
indices at least $n_0$, it strictly decreases. A non-negative integer sequence
cannot strictly decrease indefinitely, so it must reach zero. Moreover, if
$\Delta_m=0$ for some $m\ge n_0$ and $\Delta_{m+1}>0$, then applying the
contrapositive of \eqref{eq:gap-step} at $m+1$ gives
$\Delta_{m+1}<\Delta_m=0$, impossible. Thus $\Delta_n=0$ for all sufficiently
large $n$.

For $i\in [r]$ put
\[
        \Pi_i:=\prod_{j\ne i}|V_j|.
\]

The monotonicity alternative $\Delta_n\ge\Delta_{n-1}$ forces high minimum
degree in an extremal $r$-graph. Comparing this degree with the full crossing
degree at a vertex gives the useful principle that missing crossing edges must
be compensated by bad edges, up to a small error.

\begin{prop}\label{prop:degree}
Assume $\Delta_n\ge \Delta_{n-1}$. Then there is a constant $a_0>0$
such that, for every $v\in V_i$,
\[
        d_{\mathcal{H}}(v)\ge \Pi_i-a_0(\alpha^{1/r}n^{r-1}+n^{r-2}).
\]
Consequently,
\[
        d_B(v)\ge d_M(v)-a_0(\alpha^{1/r}n^{r-1}+n^{r-2}).
\]
\end{prop}

\begin{proof}
Deleting a minimum-degree vertex from the extremal $r$-graph $\mathcal{H}$ leaves an
$(n-1)$-vertex $\mathcal{F}_{r,t}$-free $r$-graph. Thus
\[
        \delta_1(\mathcal{H})\ge \ex(n,\mathcal{F}_{r,t})-\ex(n-1,\mathcal{F}_{r,t}).
\]
Using the definition of $\Delta_n$ and the assumption
$\Delta_n\ge \Delta_{n-1}$, we get
\[
        \delta_1(\mathcal{H})\ge |\mathcal{J}_{t-1}(n)|-|\mathcal{J}_{t-1}(n-1)|.
\]
A direct count gives
\[
        |\mathcal{J}_{t-1}(n)|-|\mathcal{J}_{t-1}(n-1)|
        =\frac{1}{r^{r-1}}n^{r-1}+O_{r,t}(n^{r-2}).
\]
On the other hand, \eqref{eq:part-balance} gives
\[
        \Pi_i=\frac{1}{r^{r-1}}n^{r-1}+O(\alpha^{1/r}n^{r-1}).
\]
This proves the stated lower bound on $d_{\mathcal{H}}(v)$. Finally,
\[
        d_{\mathcal{H}}(v)=\Pi_i-d_M(v)+d_B(v),
\]
so the lower bound on $d_B(v)$ follows immediately.
\end{proof}

We will pay special attention to bad edges disjoint from $W$; these are exactly
the bad edges not meeting any heavy vertex.
The next lemma shows that each such bad edge still carries enough local
structure to act like one fan in a later copy of $\mathcal F_{r,t}$.

\begin{lemma}\label{lem:bad-package}
Assume $\Delta_n\ge \Delta_{n-1}$. There exists a constant $a_{\rm miss}>0$ such that the following holds. Let $e\in \mathcal{H}\setminus \mathcal{T}_\sigma$ be a
bad edge disjoint from $W$. Choose a part $i$ with $|e\cap V_i|\ge 2$, and choose
distinct vertices $c_e,u_e\in e\cap V_i$. Then there are sets $X_j(e)\subseteq V_j\setminus e$ for $j\in [r]\setminus\{i\}$, each of size at least $\gamma_b n$, such that
\begin{equation}\label{eq:missing-ue}
        d_M(u_e)\le a_{\rm miss}(\delta+\alpha^{1/r})n^{r-1}+O_{r,t}(n^{r-2}).
\end{equation}
Moreover, if $U\subseteq V(\mathcal{H})$ satisfies
$U\cap e=\varnothing$ and $|U|\le k_{\rm aux}$, then every present crossing edge
\[
        T=\{u_e\}\cup\{x_j:j\ne i\}\in \mathcal{H}-U,
        \qquad\text{and}\qquad x_j\in X_j(e)\setminus U,
\]
extends inside $\mathcal{H}-U$ to a copy of $\mathcal{F}_r$ with center $c_e$, $i$th
transversal vertex $u_e$, $i$th petal $e$, and transversal edge $T$.
\end{lemma}

\begin{proof}
Since $e$ is disjoint from $W$, $d_B(c_e)<\delta n^{r-1}$ and $d_B(u_e)<\delta n^{r-1}$. Proposition~\ref{prop:degree} implies
\[
        d_M(c_e),d_M(u_e)
        \le a_{\rm miss}(\delta+\alpha^{1/r})n^{r-1}+O_{r,t}(n^{r-2}),
\]
which gives \eqref{eq:missing-ue}.

For $j\ne i$, define
\[
        X_j(e):=\{x\in V_j\setminus e:d_{\mathcal{H}}(c_e,x)>d\}.
\]
We show that $|X_j(e)|\ge \gamma_b n$. A cross-pair $\{c_e,x\}$ with
$x\in V_j$ lies in $\Omega_r(n^{r-2})$ crossing $r$-sets. If it is sparse,
then $\Omega_r(n^{r-2})$ of those crossing sets are missing. The collections of missing edges obtained from different vertices $x\in V_j$ are disjoint once the common vertex $c_e$ is fixed. Hence, the number of sparse vertices $x\in V_j$ is at most
\[
        O_{r,t}\left(\frac{d_M(c_e)}{n^{r-2}}\right)
        \le O_{r,t}((\delta+\alpha^{1/r})n)+O_{r,t}(1).
\]
By \eqref{eq:delta-hierarchy}, \eqref{eq:alpha-hierarchy}, and the part-size estimate
\eqref{eq:part-balance}, this
leaves at least $\gamma_b n$ vertices in $V_j\setminus e$.
Now fix $U$ and a present crossing edge
$T=\{u_e\}\cup\{x_j:j\ne i\}$ as in the statement. The fixed bad edge $e$
contains $\{c_e,u_e\}$ and meets $T$ exactly in $u_e$, because
$x_j\in V_j\setminus e$. Also, $d_{\mathcal{H}}(c_e,x_j)>d$ for every $j\ne i$. Lemma~\ref{lem:fan-completion}, applied with center $c_e$, index $i$, fixed petal $e$, and transversal edge $T$, gives the desired copy of $\mathcal{F}_r$ in $\mathcal{H}-U$.
\end{proof}

\subsection{Transversal selection and realization}

The previous subsection produces packages of candidate sets from heavy vertices
and bad edges disjoint from $W$. Here we prove that, provided the missing degrees are
small, one can choose transversal vertices from these candidate sets so that all
blowup transversal edges are present simultaneously.

Whenever Lemma~\ref{lem:bad-package} is applied to a
bad edge $e$, with repeated coordinate $i_0$ and chosen vertices
$c_e,u_e\in e\cap V_{i_0}$, we call the resulting data a \emph{bad-edge package}. If this package is assigned index $a$, we write
\[
        e_a:=e,\qquad c_a:=c_e,\qquad u_a:=u_e,
\]
and set
\[
        C_a:=e_a,\qquad X_{a,i_0}:=\{u_a\},\qquad
        X_{a,j}:=X_j(e_a)\quad\text{for }j\ne i_0.
\]
The vertex $u_a$ is called the \emph{prescribed vertex} of the package. All candidate sets other than the prescribed singletons are called
\emph{non-singleton candidate sets}.

The next selection lemma is a simple counting lemma. It chooses one transversal
vertex from each candidate set, while avoiding collisions and avoiding the
small family of missing crossing edges. The prescribed singleton candidates are
allowed because their missing degrees are separately bounded.

\begin{lemma}\label{lem:transversal-selection}
Let $1\le p\le t$, let $i_0\in[r]$, and let $A\subseteq[p]$. For each $a\in[p]$ and $j\in[r]$, let $Y_{a,j}\subseteq V_j$ be given.
Assume that, for every $a\in A$, $Y_{a,i_0}=\{u_a\}$, where the vertices $u_a\in V_{i_0}$ are distinct and satisfy $d_M(u_a)\le \zeta n^{r-1}$.
Assume also that all non-singleton candidate sets are large: $|Y_{a,j}|\ge \eta n$ for every $(a,j)\notin A\times\{i_0\}$. 
If $\alpha,\zeta>0$ are sufficiently small in terms of $r,t,\eta$ and
$n$ is sufficiently large, then one can choose vertices
$x_{a,j}\in Y_{a,j}$ such that $x_{a,j}\ne x_{b,j}$ whenever $a\ne b$ and $j\in[r]$, and such that $\{x_{a_1,1},\ldots,x_{a_r,r}\}\in \mathcal H$ for every $(a_1,\ldots,a_r)\in[p]^r$.
\end{lemma}
\begin{proof}
Let
$$
         I_{\rm tr}:=([p]\times [r])\setminus (A\times\{i_0\}),
        \qquad N:=|I_{\rm tr}|=pr-|A|.
$$
For $a\in A$ set $x_{a,i_0}:=u_a$. We choose the remaining variables
$x_{a,j}$, $(a,j)\in I_{\rm tr}$. There are at least $\eta^N n^N$ assignments. Call an assignment colliding if two selected vertices in the same part
coincide. This includes coincidences between a non-prescribed vertex in
$V_{i_0}$ and one of the prescribed vertices $u_a$. Since $p\le t$ and
all non-singleton candidate sets have size at least $\eta n$, the number of colliding assignments is $O_{r,t}(n^{N-1})$.
Fix $\mathbf a=(a_1,\ldots,a_r)\in[p]^r$ and put
$
        T_{\mathbf a}:=\{x_{a_1,1},\ldots,x_{a_r,r}\}.
$
Apart from the colliding assignments already counted, the set
$T_{\mathbf a}$ contains a prescribed vertex if and only if
$a_{i_0}\in A$.

First suppose $a_{i_0}\notin A$. Then all $r$ vertices of
$T_{\mathbf a}$ are chosen from non-singleton candidate sets. The number
of assignments for which $T_{\mathbf a}$ is missing is at most
$$
        |\mathcal T_\sigma\setminus\mathcal H|\,O_{r,t}(n^{N-r})
        =
        O_{r,t}(\alpha n^N).
$$
Now suppose $a_{i_0}\in A$. Then $T_{\mathbf a}$ contains the prescribed
vertex $u_{a_{i_0}}$. The missing-degree bound at this vertex gives at
most
$$
        O_{r,t}\bigl(d_M(u_{a_{i_0}})n^{N-r+1}\bigr)
        =
        O_{r,t}(\zeta n^N)
$$
assignments for which $T_{\mathbf a}$ is missing.
There are at most $t^r$ choices of $\mathbf a$. Hence the total number
of forbidden assignments is at most
$$
        O_{r,t}(n^{N-1})
        +O_{r,t}(\alpha n^N)
        +O_{r,t}(\zeta n^N).
$$
For $\alpha$ and $\zeta$ sufficiently small in terms of $r,t,\eta$, and
then for $n$ sufficiently large, this is smaller than
$\eta^N n^N$. Therefore some assignment is non-colliding and makes every
required transversal edge present. This assignment satisfies the desired
conclusion.
\end{proof}

Once the transversal vertices have been chosen, the package definitions provide
the individual fans. The assembly lemma records this final step: $t$ compatible
packages force a copy of the whole semi-blowup fan.

% \begin{lemma}\label{lem:assembly}
% Suppose that packages indexed by $a\in[t]$ are given, each of which is either a heavy package or a bad-edge package. Assume that the seed sets $C_a$ are pairwise disjoint, that all bad-edge packages have the same repeated coordinate $i_0$, that every non-singleton candidate set has size at least $\eta n$, and that every prescribed vertex $u_a$ satisfies $ d_M(u_a)\le \zeta n^{r-1}$,  where $\alpha,\zeta\ll_{r,t,\eta}1$. Then $\mathcal H$ contains a copy of $\mathcal F_{r,t}$.
% \end{lemma}
 
\begin{lemma}\label{lem:assembly}
Suppose that packages indexed by $a\in[t]$ are given, each of which is either a heavy package or a bad-edge package. Assume that $|R_a|\le k_{\rm aux}/2$ whenever $a$ indexes a heavy package, that the seed sets $C_a$ are pairwise disjoint, that all bad-edge packages have the same repeated coordinate $i_0$, that every non-singleton candidate set has size at least $\eta n$, and that every prescribed vertex $u_a$ satisfies $ d_M(u_a)\le \zeta n^{r-1}$,  where $\alpha,\zeta\ll_{r,t,\eta}1$. Then $\mathcal H$ contains a copy of $\mathcal F_{r,t}$.
\end{lemma}

\begin{proof}
Put
\[
        C_{\rm seed}:=\bigcup_{a=1}^t C_a.
\]
For each bad-edge package $a$, keep the singleton candidate set
$Y_{a,i_0}:=\{u_a\}$. For all other pairs $(a,j)$, set
\[
        Y_{a,j}:=X_{a,j}\setminus C_{\rm seed}.
\]
Since $|C_{\rm seed}|=O_{r,t}(1)$, every non-singleton set $Y_{a,j}$
has size at least $\eta n/2$ for all sufficiently large $n$. The
prescribed vertices $u_a$ are distinct, because the seed sets are
pairwise disjoint. Applying
Lemma~\ref{lem:transversal-selection}, with parameter $\eta/2$, $p=t$, and
$A$ equal to the set of bad-edge indices, gives vertices $x_{a,j}\in Y_{a,j}$
such that the selected vertices are distinct within each part and
\[
        \{x_{a_1,1},\ldots,x_{a_r,r}\}\in \mathcal H
        \qquad\text{for all }(a_1,\ldots,a_r)\in[t]^r.
\]
We embed the fans in the order $a=1,\dots,t$. At step $a$, let $U_a$ be the
union of the vertices of the fans already embedded, the later seed sets
$C_b$ with $b>a$, and the later selected transversal vertices
$x_{b,j}$ with $b>a$ and $j\in[r]$. Then $U_a\cap C_a=\varnothing$. Indeed,
earlier fans were embedded while avoiding the later seed sets, the seed sets
are pairwise disjoint, and every later selected vertex either avoids
$C_{\rm seed}$ or is the prescribed vertex in its own seed set. Also $|U_a|\le k_{\rm aux}/2$, and, in the heavy case, $|R_a\cup U_a|\le k_{\rm aux}$.
% Also
% $|U_a|\le k_{\rm aux}$, and, in the heavy case,
% $|R_a\cup U_a|\le k_{\rm aux}$.
If $a$ is of heavy type, then $T_a:=\{x_{a,1},\ldots,x_{a,r}\}$
is an edge of $\mathcal H-(R_a\cup U_a)$ and satisfies
$x_{a,j}\in X_{a,j}\setminus U_a$ for every $j\in[r]$. Hence
Proposition~\ref{prop:heavy-package}, applied with its original set $R_a$
and with forbidden set $U_a$, embeds the $a$th fan.
If $a$ is of bad-edge type, then $T_a:=\{u_a\}\cup\{x_{a,j}:j\ne i_0\}$
is an edge of $\mathcal H-U_a$ and satisfies
$x_{a,j}\in X_j(e_a)\setminus U_a$ for every $j\ne i_0$. Since
$U_a\cap e_a=\varnothing$, Lemma~\ref{lem:bad-package} embeds the $a$th fan
with fixed petal $e_a$ and transversal edge $T_a$.
The fans are vertex-disjoint by the choice of the forbidden sets $U_a$, and the
selected vertices provide all blowup transversal edges. Hence these fans form
a copy of $\mathcal F_{r,t}$.
\end{proof}

\subsection{Bad edges outside the heavy vertices}

It remains to show that bad edges disjoint from the heavy vertices cannot form
a large structure. A matching of such bad edges would provide enough
bad-edge packages to assemble a forbidden semi-blowup fan.

Let $\nu_B$ be the maximum size of a matching of bad edges disjoint from $W$.

\begin{lemma}\label{lem:outside-matching}
Assume $\Delta_n\ge \Delta_{n-1}$. Then the maximum size $\nu_B$ of a matching of
bad edges disjoint from $W$ satisfies $\nu_B<rt$.
\end{lemma}

\begin{proof}
Suppose there is a matching of size at least $rt$ consisting of bad edges
disjoint from $W$. Each such bad edge has at least one repeated part. By the
pigeonhole principle, there are $t$ pairwise disjoint bad edges disjoint from
$W$, say $e_1,\dots,e_t$, with the same repeated coordinate $i_0$. For each $a$, choose distinct vertices
$c_a,u_a\in e_a\cap V_{i_0}$ and apply Lemma~\ref{lem:bad-package}. Thus
the $t$ indices are of bad-edge type, all with repeated coordinate $i_0$; their
non-singleton candidate sets have size at least $\gamma_b n$, and the prescribed
vertices satisfy
\eqref{eq:missing-ue}. Equivalently, their missing degrees are at most
$\zeta_n n^{r-1}$ with
\[
        \zeta_n=a_{\rm miss}(\delta+\alpha^{1/r})+O_{r,t}(n^{-1}).
\]
By the hierarchy and the choice of $n$, we have
$\zeta_n\ll_{r,t,\gamma_b}1$. Lemma~\ref{lem:assembly}
now gives a copy of $\mathcal F_{r,t}$, a contradiction.
\end{proof}

\begin{cor}\label{cor:outside-count}
Assume $\Delta_n\ge \Delta_{n-1}$, and let $\mathcal B_0$ be the family of bad
edges disjoint from $W$. Then
\[
        |\mathcal B_0|\le c_{\rm bad}\delta n^{r-1}.
\]
\end{cor}

\begin{proof}
Let $e_1,
\dots,e_s$ be a maximum matching of bad edges disjoint from $W$ and put
$U=e_1\cup\cdots\cup e_s$. By Lemma~\ref{lem:outside-matching}, $s<rt$.
Every bad edge disjoint from $W$ meets $U$; otherwise, the matching was not maximal.
Since every vertex of $U$ lies outside $W$, $d_B(u)<\delta n^{r-1}$ for all
$u\in U$. Therefore
\[
        |\mathcal B_0|\le \sum_{u\in U}d_B(u)
        \le |U|\delta n^{r-1}\le c_{\rm bad}\delta n^{r-1}.
\]
\end{proof}

\subsection{Proof of the exact Tur\'an theorem}

We now combine the preceding estimates to prove the exact result for
$r\ge3$. The argument first forces exactly $t-1$ heavy vertices, then rules out
all bad edges disjoint from the heavy vertices, and finally identifies the extremal $r$-graph with
$\mathcal J_{t-1}(n)$.

\begin{proof}[Proof of Theorem~\ref{thm:exact} for $r\ge 3$]
It remains to prove the implication \eqref{eq:gap-step}. Fix $n$ sufficiently large and assume
$\Delta_n\ge \Delta_{n-1}$. Let $\mathcal H$ be an extremal $\mathcal{F}_{r,t}$-free $r$-graph on
$n$ vertices and put
$w:=|W|$.
By Corollary~\ref{cor:heavy-bound}, $w\le t-1$.

Let $\mathcal B_0$ be the family of bad edges disjoint from $W$. By
Corollary~\ref{cor:outside-count},
$|\mathcal B_0|\le c_{\rm bad}\delta n^{r-1}$.
Let $\mathcal{J}(W,\sigma)$ be the $r$-graph obtained by taking all $r$-sets meeting $W$
and all crossing $r$-sets with respect to the partition
$(V_1\setminus W,\dots,V_r\setminus W)$. Every edge of $\mathcal H$ either belongs to
$\mathcal{J}(W,\sigma)$ or is a bad edge disjoint from $W$. Hence
\[
        |\mathcal{H}|\le |\mathcal{J}(W,\sigma)|+c_{\rm bad}\delta n^{r-1}.
\]
The number of edges in $\mathcal{J}(W,\sigma)$ is maximized when the parts
$V_1\setminus W,\ldots,V_r\setminus W$ are balanced, so
\[
        |\mathcal{J}(W,\sigma)|\le |\mathcal{J}_w(n)|.
\]
If $w\le t-2$, then Lemma~\ref{lem:deterministic} gives
\[
        |\mathcal{J}_{t-1}(n)|-|\mathcal{J}_w(n)|\ge c_{r,t}n^{r-1}.
\]
By the hierarchy \eqref{eq:delta-hierarchy}, the constant $\delta$ was chosen
so small that $c_{\rm bad}\delta < c_{r,t}/2$.
Hence, if $w\le t-2$, then
\[
        |\mathcal{H}|
        \le |\mathcal{J}_w(n)|+c_{\rm bad}\delta n^{r-1}
        < |\mathcal{J}_{t-1}(n)|,
\]
contradicting Lemma~\ref{lem:deterministic}. Therefore
$w=t-1$.

It remains to show that $\mathcal B_0=\varnothing$. Suppose not, and choose
$e\in \mathcal B_0$. Let
\[
        \mathcal M_{\mathcal{J}}:=\mathcal{J}(W,\sigma)\setminus \mathcal{H}.
\]
Since all edges of $\mathcal H$ outside $\mathcal{J}(W,\sigma)$ lie in $\mathcal B_0$,
\[
        |\mathcal{H}|=|\mathcal{J}(W,\sigma)|-|\mathcal M_{\mathcal{J}}|+|\mathcal B_0|
        \le |\mathcal{J}_{t-1}(n)|-|\mathcal M_{\mathcal{J}}|+|\mathcal B_0|.
\]
But $|\mathcal{H}|=\ex(n,\mathcal{F}_{r,t})\ge |\mathcal{J}_{t-1}(n)|$, and therefore
\begin{equation}\label{eq:MJ}
        |\mathcal M_{\mathcal{J}}|\le |\mathcal B_0|\le c_{\rm bad}\delta n^{r-1}.
\end{equation}
We first show that every vertex $v\in W$ satisfies
$d_B(v)\ge \lambda_r n^{r-1}$. Let $v\in W\cap V_i$.
Since $w=t-1$, every $r$-set containing $v$ is allowed in $\mathcal{J}(W,\sigma)$. By
\eqref{eq:MJ},
\[
        d_{\mathcal{H}}(v)\ge \binom{n-1}{r-1}-|\mathcal M_{\mathcal{J}}|.
\]
The number of crossing $r$-sets through $v$ is at most
\[
        \Pi_i=\prod_{j\ne i}|V_j|
        =\frac{1}{r^{r-1}}n^{r-1}+O(\alpha^{1/r}n^{r-1}).
\]
Thus, by \eqref{eq:delta-hierarchy} and \eqref{eq:alpha-hierarchy}, for all
sufficiently large $n$,
\[
        d_B(v)
        \ge
        \binom{n-1}{r-1}-\Pi_i-|\mathcal M_{\mathcal{J}}|
        \ge \lambda_r n^{r-1}.
\]
Since $e$ is a bad edge, choose a coordinate $i_0$ with
$|e\cap V_{i_0}|\ge2$ and distinct vertices $c_e,u_e\in e\cap V_{i_0}$.
Applying Lemma~\ref{lem:bad-package} to $e$, we make index $t$ a bad-edge package:
\[
        C_t=e,\qquad X_{t,i_0}=\{u_e\},\qquad
        X_{t,j}=X_j(e)\quad\text{for }j\ne i_0.
\] 
Enumerate $W=\{w_1,\dots,w_{t-1}\}$. We choose the heavy indices
recursively. Once $C_1,\dots,C_{a-1}$ have been chosen, set
\[
        R_a:=e\cup\Bigl(\bigcup_{b<a}C_b\Bigr)\cup\{w_b:b>a\}.
\]
Since $e$ is disjoint from $W$ and earlier seed sets were chosen avoiding all later vertices
$w_b$, we have $w_a\notin R_a$; also $|R_a|\le k_{\rm aux}/2$.
% also $|R_a|\le k_{\rm aux}$.
Applying
Proposition~\ref{prop:heavy-package} with $\mu=\lambda_r$ to
$w_a$ and $R_a$ gives the $a$th heavy package:
\[
        C_a=C(w_a;R_a),\qquad X_{a,j}=X_j(w_a;R_a)\quad\text{for }j\in[r].
\]
Thus we have $t-1$ heavy packages and one bad-edge package. Their seed
sets are pairwise disjoint, every non-singleton candidate set has size
at least $\gamma_h(\lambda_r)n$ or $\gamma_b n$, and the only prescribed
vertex $u_e$ satisfies
\[
        d_M(u_e)\le \zeta_n n^{r-1},
        \qquad
        \zeta_n:=a_{\rm miss}(\delta+\alpha^{1/r})+O_{r,t}(n^{-1}).
\]
Let
\[
        \eta:=\eta_*=\frac12\min\{\gamma_h(\lambda_r),\gamma_b\}.
\]
By the hierarchy and the choice of $n$, we have
$\zeta_n\ll_{r,t,\eta}1$. Lemma~\ref{lem:assembly} gives a copy
of $\mathcal F_{r,t}$ in $\mathcal H$, a contradiction. Hence
$\mathcal B_0=\varnothing$.

Therefore every edge of $\mathcal H$ lies in $\mathcal{J}(W,\sigma)$. Since $|W|=t-1$,
\[
        |\mathcal{H}|\le |\mathcal{J}(W,\sigma)|\le |\mathcal{J}_{t-1}(n)|.
\]
The reverse inequality follows from Lemma~\ref{lem:deterministic}, so equality holds
throughout. Hence the parts $V_1\setminus W,\ldots,V_r\setminus W$ are balanced and every edge allowed by the
join is present. Consequently
$\mathcal{H}\cong \mathcal{J}_{t-1}(n)$.
Thus $\Delta_n=0$. As explained above, this proves the exact value and
uniqueness for all sufficiently large $n$.
\end{proof}

\section{Concluding remarks}\label{sec:remarks}

Theorem~\ref{thm:main} shows that stability does not control the
supersaturation function up to any fixed constant factor. The failure already occurs at the first
supersaturation step $q=1$, and is not caused by non-uniqueness or instability
of the extremal construction.

This leaves open the problem of identifying structural hypotheses that rule out
the local concentration phenomenon exploited here. More precisely, one needs
conditions ensuring that the copies created by a single added edge are
sufficiently spread inside the extremal construction, so that deleting a small
compensating family of old edges cannot destroy a positive proportion of them.

\begin{problem}
Find natural structural conditions on a stable non-$r$-partite $r$-graph
$\mathcal F$ which guarantee the following asymptotic local supersaturation
bound: there exist $\delta>0$ and a sequence
$\varepsilon_n\to0$ such that, for all sufficiently large $n$ and all
$1\le q\le \delta n$,
\[
        h_{\mathcal F}(n,q)\ge (1-\varepsilon_n)q\,c(n,\mathcal F).
\]
\end{problem}

\section*{Acknowledgments}
Heng Li was supported by the National Natural Science Foundation of China (12501487), the China Scholarship Council, and the Institute for Basic Science (IBS-R029-C4).
Hong Liu was supported by the Institute for Basic Science (IBS-R029-C4).
Xizhi Liu was supported by the Excellent Young Talents Program (Overseas) of the National Natural Science Foundation of China.
Jing Wang was supported by the National Natural Science Foundation of China (12301437) and the Institute for Basic Science (IBS-R029-C4).

\section*{Data Availability Statement}
Data sharing is not applicable to this article as no new data were created or analyzed in this study.

\section*{Declaration on the use of AI}
AI tools were used solely to improve the language and presentation of the manuscript. All mathematical ideas and proofs in this paper were developed independently by the authors, who take full responsibility for the content of the paper.

\bibliographystyle{fmsigma}
\bibliography{ref-sigma}

\end{document}